\documentclass[12pt,a4paper]{amsart}

\usepackage{hyperref}
\usepackage{paralist, mathdots} 

\usepackage{amssymb} 
\usepackage{amsmath,tikz,wasysym,multirow,enumitem,lscape,multirow,subfigure}
\usepackage[section]{algorithm} 

\setlist[enumerate,1]{label = (\alph*),ref = (\alph*)}

\usepackage{mwe}
\usepackage{hyperref} 
\usepackage{amsthm}
\usepackage{mathtools}
\usepackage{soul, arydshln}
\usepackage[noadjust]{cite}\usepackage{MnSymbol}

\tikzset{every node/.style={draw, circle, inner sep=2pt},
every picture/.append style={thick,scale=0.8},
every label/.style={draw=none, rectangle}}

\newtheorem{proposition}{Proposition}[section]

\newtheorem{theorem}[proposition]{Theorem}

\newtheorem{corollary}[proposition]{Corollary}
\newtheorem{lemma}[proposition]{Lemma}
\theoremstyle{definition}
\newtheorem{example}[proposition]{Example}
\newtheorem{remark}[proposition]{Remark}
\newtheorem{definition}[proposition]{Definition}

\newcommand{\npmatrix}[1]{\left( \begin{matrix} #1 \end{matrix} \right)}
\newcommand{\R}{\mathbb{R}}

\newcommand{\bR}{\mathbb{R}}
\newcommand{\bN}{\mathbb{N}}

\let\C\relax
\newcommand{\+}{\phantom{+}}
\newcommand{\C}{\mathbb{C}}

\newcommand{\mc}{\mathcal}
\DeclareMathOperator{\iso}{iso}
\newcommand{\isa}[1]{\iota(#1)}

\DeclareMathOperator{\rk}{rank}
\DeclareMathOperator{\tr}{tr}  
  
\DeclareMathOperator{\mlt}{mult}
\DeclareMathOperator{\rank}{rank}
\DeclareMathOperator{\diag}{diag}
\DeclareMathOperator{\diam}{diam}

\newcommand{\1}{{\bf 1}}
\newcommand{\e}{{\bf e}}

\newcommand{\bbb}{{\bf b}}
\newcommand{\uu}{{\bf u}}
\newcommand{\ww}{{\bf w}}
\newcommand{\xx}{{\bf x}}
\newcommand{\yy}{{\bf y}}

\newcommand{\tuple}{\mathbf}
\newcommand{\m}{\tuple{m}}
\newcommand{\n}{\tuple{n}}
\newcommand{\p}{\tuple{p}}
\newcommand{\vv}{\tuple{v}}
\newcommand{\kx}[1]{K_{#1}}

\newcommand{\km}{\kx{\m}}
\newcommand{\kn}{\kx{\n}}

\newcommand{\trans}{^\top}

    \definecolor{helena}{rgb}{.2,.8,.4}
    \definecolor{polona}{rgb}{.2,.2,.8}
    \definecolor{rupert}{rgb}{0,.5,.5}
   \definecolor{todo}{rgb}{.8,.2,.2}

\AtBeginDocument{%
   \def\MR#1{}
}

\begin{document}
\title[Orthogonal symmetric matrices and joins of graphs]
{Orthogonal symmetric matrices \\
and joins of graphs}
\author{Rupert H. Levene, Polona Oblak, Helena \v Smigoc}
\address[R.~H.~Levene and H.~\v Smigoc]{School of Mathematics and Statistics, University College Dublin, Belfield, Dublin 4, Ireland}
\email{rupert.levene@ucd.ie}\email{helena.smigoc@ucd.ie}
\address[P.~Oblak]{Faculty of Computer and Information Science, University of Ljubljana, Ve\v cna pot 113, SI-1000 Ljubljana, Slovenia}
\email{polona.oblak@fri.uni-lj.si}

\bigskip

\begin{abstract}
  We introduce a notion of compatibility for multiplicity matrices. This gives rise to a necessary condition for the join of two (possibly disconnected) graphs $G$ and $H$ to be the pattern of an orthogonal symmetric matrix, or equivalently, for the minimum number of distinct eigenvalues $q$ of $G\vee H$ to be equal to two. 
Under additional hypotheses, we show that this necessary condition is also sufficient.
  As an application, we prove that $q(G\vee H)$ is either two or three when $G$ and $H$ are unions of complete graphs, and we characterise when each case occurs.

\end{abstract}

\maketitle

\section{Introduction}

\subsection{Background and related work}
Orthogonal matrices, ubiquitous in  matrix theory and central in applications, have been widely studied. However, some basic questions on their combinatorial structure remain open \cite{MR2002914,MR4057549}. 
In this paper we advance our understanding of the zero-nonzero patterns of symmetric orthogonal matrices, drawing motivation from the Inverse Eigenvalue Problem for Graphs, a more general problem on the interplay between spectral and structural properties of a matrix. 


Let $G$ be a simple graph with vertex set $V(G)=\{1,\dots,n\}$ and edge set~$E(G)$, and consider $S(G)$, the set of all real symmetric $n \times n$ matrices $A=(a_{ij})$ such that, for $i \neq j$, $a_{ij} \neq 0$ if and only if $\{i,j\} \in E(G)$, with no restriction on the diagonal entries of $A$. The \emph{Inverse Eigenvalue Problem for Graphs (IEP-G)} seeks to characterise all possible sets of eigenvalues of matrices in $S(G)$. The IEP-G is unsolved for all except a few families of graphs, and it has motivated the study of several related parameters. One widely studied parameter is the \emph{minimum number of distinct eigenvalues of a graph},
namely $q(G) = \min\{q(A)\colon A \in S(G)\}$,
where $q(A)$ denotes the number of distinct eigenvalues of a square matrix~$A$.  For a graph $G$ the set $S(G)$ contains an orthogonal symmetric matrix (or $G$ is \emph{is realisable by} an orthogonal symmetric matrix) if and only if $q(G)=2$.

The study of the minimum number of distinct eigenvalues of a graph was initiated by Leal-Duarte and Johnson in \cite{MR1899084}, where they proved that if $T$ is a tree, then $q(T)$ cannot be smaller than $\diam(T)+1$; this bound was later improved for infinitely many trees (see e.g.,~\cite{MR2735867}).
Determining $q$ for many graphs with cycles seems to be a difficult problem. It is well known that $q(G)=n$ if and only if $G$ is a path \cite{MR0364288}, and graphs with $q(G)=n-1$ were characterised in \cite{MR3665573} using strong spectral properties. At the other extreme, while it is clear that $q(G)=1$ if and only if $E(G)=\emptyset$, there is no known characterisation of graphs with $q(G)=2$, i.e., the graphs that are realisable by an orthogonal symmetric matrix.

Several families of graphs are known to have a realisation with an orthogonal symmetric matrix 
(see, e.g.,~\cite{MR3904092,MR3118943,MR2470115,MR3891770}), 
 and several necessary conditions 
were determined by Adm et.~al.~\cite{MR3118943,corrigendum}. The same authors also proved that $q(G \vee G)=2$  for any connected graph $G$. Later, this result was generalized by Monfared and Shader \cite{MR3506498}, who proved that $q(G \vee H)=2$ for any connected graphs $G$ and $H$ with the same number of vertices. Recently, joins of disconnected graphs were investigated in \cite{MR4044603,corrigendum}, where particular attention was given to joins of unions of complete graphs.

\subsection{Overview}
Results in this paper contribute to the IEP-G, to the study of combinatorial structure of orthogonal matrices, and also shed light on certain completion problems for orthogonal matrices. 

The main framework is developed in Section~\ref{sec:compatibility}, where we present a necessary condition for the join of two (possibly disconnected) graphs, $G \vee H$, to have $q(G \vee H)=2$, as well as a closely related sufficient condition. Section~\ref{sec:examples} contains remarks, extensions of previous results and initial examples derived from results in Section~\ref{sec:compatibility}.  Finally, in Section~\ref{sec:complete} we apply our results to determine $q(G \vee H)$ whenever $G$ and $H$ are unions of complete graphs. As a preview of our results, we state the main theorem of Section~\ref{sec:complete} below. 
%
%

\begin{theorem}\label{thm:main}
  Let $G=G_1\cup\dots\cup G_k$ and $H=H_1\cup\dots\cup H_\ell$ where $G_1,\dots,G_k$ and $H_1,\dots,H_\ell$ are connected graphs 
  and $k\le \ell$. By $\iso(G)$ we denote the set of isolated vertices of $G$. If at least one of the following three conditions holds, then $G \vee H$ is not realisable by an orthogonal symmetric matrix (i.e.,~$q(G\vee H)\ne 2$).
\begin{enumerate}
\item $|G|<\ell$;
 \item $\iso(G)\ne \emptyset$  and $k+\ell>|G|+|\iso(G)|$;
\item $\iso(G)=\emptyset$ and $\iso(H)\ne \emptyset$, and $k<\ell<2k$ and $|H|<2k$ and $|G|<k+\ell$.
\end{enumerate}
Moreover, if the connected components of $G$ and $H$ are all complete graphs, then
\[q(G\vee H)=
  \begin{cases}
    2&\text{if none of conditions (a), (b), (c) is satisfied,}\\
    3&\text{if at least one of (a), (b), (c) is satisfied.}
  \end{cases}\]
\end{theorem}

\subsection{Notation}

 We denote by $\bN=\{1,2,\dots\}$ and $\bN_0=\{0\}\cup \bN$ the sets of positive and nonnegative integers. Moreover, let  $[n]=\{1,2,\dots,n\}$ and $m+[n]=\{m+1,m+2,\dots,m+n\}$.

 For any set $\mathbb{S}$, we denote by $\mathbb{S}^{m\times n}$ the set of $m\times n$ matrices with entries in $\mathbb{S}$, and let $\mathbb{S}^m$ denote $\mathbb{S}^{m\times 1}$ or $\mathbb{S}^{1\times m}$, depending on the context. For a matrix $X\in \R^{m \times n}$, the transpose of $X$ is denoted by $X\trans$, and we write $X>0$ 
 if every entry of $X$ is greater than~$0$.  Similarly, for matrices $X=(x_{ij})$ and $Y=(y_{ij})$ in $\R^{m \times n}$, we write $X \geq Y$ if $x_{ij} \geq y_{ij}$ for all $i \in [m]$, $j \in [n]$.
 We say that a matrix $X$ is nowhere-zero if no entry of $X$ is zero. We write $X\ne 0$ if at least one entry of $X$ is nonzero.

The following notation is used for special vectors and matrices: $\1_m$ is the column vector of ones in $\bR^{m\times 1}$, $\e_i:=\npmatrix{0,\dots,0,1,0,\dots,0}\trans $ is the vector with the $1$ in the $i$th entry and zeros elsewhere, $I_k$ is the $k \times k$ identity matrix, $0_{m \times n}$ is the zero $m \times n$ matrix, and we also write $0_m:=0_{m\times m}$ and $\mathbf{0}_m:=0_{m\times 1}$. (We allow any of $k,m,n$ to be zero, in which case the corresponding matrix is empty.)

If $X\in \R^{m \times n}$ and $\mc R\subseteq [m]$,  $\mc C\subseteq [n]$, then  $X[\mc R,\mc C]$ is the submatrix of $X$ with rows $\mc R$ and columns $\mc C$. Let $\bigoplus_{i \in [r]} A_i$ denote the direct sum of matrices $A_1,\ldots,A_r$, where, for technical reasons, we allow for the possibility that $A_i$ is an empty matrix. We denote the the diagonal matrix with diagonal entries $\Lambda=(\lambda_1,\dots,\lambda_r)\in \bR^r$ occurring with multiplicities $\vv=(v_1,\ldots,v_r)\trans\in \bN_0^r$ by 
$ D_{\Lambda,\vv}:=\bigoplus_{i\in [r]}\lambda_i I_{v_i}$. Since $\Lambda$ lists the eigenvalues of this matrix (ignoring multiplicities), we sometimes refer to $\Lambda$ as an eigenvalue list.
In the case $\vv=\1_r$ (i.e., all multiplicities are equal to one), we abbreviate this diagonal matrix as $\diag(\lambda_1,\dots,\lambda_r)$.

 Let $\sigma(A)$ denote the multiset of eigenvalues of a square matrix $A$, counted with algebraic multiplicities, and let  $q(A)$ denote  the number of distinct eigenvalues of $A$. By  $\mlt(\lambda,A)\in \bN_0$ we denote the multiplicity of a real number $\lambda$ in  $\sigma(A)$.

In this paper, all graphs are simple undirected graphs with a non-empty vertex sets. For a graph $G=(V(G),E(G))$, the \emph{order} of $G$ is denoted by $|G|:=|V(G)|$. A \emph{connected component} of $G$ is a maximal subgraph of $G$ in which any two vertices are connected via a path. The set of \emph{isolated vertices} of $G$, i.e.,~the set of vertices of degree zero, is denoted by  $\iso(G)$.
The \emph{join} $G \vee H$ of two graphs $G$ and $H$ is the disjoint graph union $G \cup H$ together with all the possible edges joining the vertices in $G$ to the vertices in $H$. We abbreviate the disjoint graph union of $k$ copies of the same graph $G$ by $kG:=G\cup\dots\cup G$. 
We write $P_n$, $C_n$ and $K_n$ for the \emph{path}, the \emph{cycle} and the \emph{complete graph} on $n$ vertices, respectively,
and we denote the \emph{complete bipartite graph} on two disjoint sets of cardinalities $m$ and $n$ by $K_{m,n}:=mK_1 \vee nK_1$.

Recall that for a graph $G$ of order $n$, with $V(G)$ identified with $[n]$, we write
$$S(G):=\{A=(a_{ij}) \in \R^{n \times n}\colon \; a_{ij}\ne 0 \iff \{i,j\} \in E(G) \text{ for } i \ne j \}$$
and
$q(G) := \min\{q(A)\colon A \in S(G)\}$.
For $A\in S(G)$, the matrix $A[H]$ is the principal submatrix of $A$ whose rows and columns are the vertices of a subgraph $H\subseteq G$.

\section{Compatible multiplicity matrices}\label{sec:compatibility}


Below we introduce the notion of compatible multiplicity matrices, which will give a necessary condition for $q(G \vee H)=2$.

\begin{definition}\label{def:mult-vector}
  Let $G$ be a connected graph, $n:=|G|$ and $r \in \bN$. We call a vector $\vv=\npmatrix{v_1 & v_2&\ldots &v_{r} }\trans  \in \bN_0^{r}$ 
  \emph{a multiplicity vector} for $G$ if $\vv$ is an ordered multiplicity list that can be realised by a matrix in $S(G)$. In other words,  $\sum_{i=1}^r v_i=n$ and there is an eigenvalue list $\Lambda=(\lambda_1,\dots,\lambda_r)\in \bR^r$ 
  with $\lambda_i < \lambda_{i+1}$ for $i\in [r-1]$, and an orthogonal matrix $U\in \R^{n \times n}$, so that $UD_{\Lambda,\vv} U\trans \in S(G)$.
%
\end{definition}


\begin{definition}
  Let $G$ be a graph and $r,k\in\bN$. We say that a matrix $V\in \bN_0^{r \times k}$ is a \emph{multiplicity matrix} for $G$ if $G$ has $k$ connected components $G_1,\dots,G_k$, and for $1\le i\le k$, the $i$th column of $V$ is a multiplicity vector for $G_i$.
\end{definition}

We note for future reference that $|G|=\1_r\trans V\1_k$ whenever $V$ is an $r\times k$ multiplicity matrix for a graph~$G$.





\begin{definition}
  For a matrix $X$ with at least $3$ rows, we write $\widetilde X$ for the matrix obtained by deleting the first row and the last row of $X$. Let $r,k,\ell\in \bN$ with $r\ge3$. Two matrices $V \in \bN_0^{r \times k}$ and $W \in \bN_0^{r \times \ell}$ are said to be \emph{compatible}
  if $\widetilde V\1_k=\widetilde W\1_\ell$ and $\widetilde V\trans \widetilde W>0$. We say that two graphs $G,H$ \emph{have compatible multiplicity matrices} if there exist compatible matrices $V,W$ where $V$ is a multiplicity matrix for $G$ and $W$ is a multiplicity matrix for $H$.
\end{definition}

In Theorem \ref{thm:q=2} we will show that compatibility of multiplicity matrices is a necessary condition for the join of two graphs to have $q=2$. For this, we use the following lemma on the Sylvester equation (see, e.g.,~\cite{MR2978290}), whose simple proof we omit. 

\begin{lemma}\label{lem:Sylv}
Suppose $A \in \R^{n \times n}$, $B \in \R^{n \times m}$ and $C \in \R^{m \times m}$ satisfy the equation $AB-BC=0$. If $\vv$ is an eigenvector of $C$ corresponding to the eigenvalue $\lambda$, then $B\vv$ is either equal to zero, or is an eigenvector of $A$ corresponding to the same eigenvalue, $\lambda$. In particular, if $B$ has trivial kernel, then $\sigma(C)\subseteq\sigma(A)$ and hence $m\le n$.
\end{lemma}

\begin{theorem}\label{thm:q=2}
Let $G$ and $H$ be two graphs. If $q(G\vee H)=2$, then $G$ and $H$ have compatible multiplicity matrices.
\end{theorem}

\begin{proof}
Decompose $G$ and $H$ into their connected components as $G=\bigcup_{i\in [k]} G_i$ and $H=\bigcup_{j\in [\ell]} H_j$. Since $q(G\vee H)=2$,  there is an orthogonal symmetric matrix $$X:=
\npmatrix{
A & B \\ B\trans  & -C
}\in S(G \vee H)$$ 
where $A:=\bigoplus_{i\in [k]} A_i$ with $A_i\in S(G_i)$, $C:=\bigoplus_{j\in [\ell]} C_j$ with $C_j\in S(H_j)$, and $B$ is a nowhere-zero matrix. Since $X$ is orthogonal and symmetric, we have $X^2=I$.

Let us have a closer look at eigenvalues of $A$ and $C$. First we note that if $Y\in \{A_1,\dots,A_k,C_1,\dots,C_\ell\}$, then $Y$ is a principal submatrix of the orthogonal symmetric matrix~$X$, so every eigenvalue of $Y$ is in the closed interval $[-1,1]$. Moreover, at least one eigenvalue of $Y$ is in the open interval $(-1,1)$; for otherwise, $Y$ is orthogonal and from $X^2=I$ it follows easily that some row or column of $B$ must be zero, contrary to hypothesis.

Therefore $A\oplus C$ has at least one eigenvalue in the interval $(-1,1)$. Write the distinct eigenvalues of $A\oplus C$ in $(-1,1)$ as $\lambda_2<\dots<\lambda_{r-1}$, where $r\in\bN$ with $r\ge3$, and define  $\lambda_1=-1$ and $\lambda_r=1$.  Let $V$ be the multiplicity matrix for $G$ realised by $A$, i.e., $V=(v_{si})_{s,i}\in \bN_0^{r \times k}$ where $v_{si}:=\mlt(\lambda_{s},A_i)$ for $s\in [r]$ and $i\in [k]$.
Similarly, let $W=(w_{sj})_{s,j}\in \bN_0^{r\times \ell}$  be the multiplicity matrix for $H$ realised by $C$, with $w_{sj}:=\mlt(\lambda_{s},C_j)$. We proceed to show that $V$ and $W$ are compatible.

Let $U_A:=\bigoplus_{i\in [k]}  U_{A_i}$ and $U_C:=\bigoplus_{j\in [\ell]}  U_{C_j}$ where $U_{A_i}$ and $U_{C_j}$ are orthogonal matrices which diagonalise $A_i$ and $C_j$, respectively.
Consider the orthogonal symmetric matrix $X':=(U_A\oplus U_C)\trans X(U_A\oplus U_C)$. Since $X'$ is orthogonal, each of its rows and columns is a unit vector. In particular,  any row or column of $X'$ with $\pm1$ on the diagonal has every other entry equal to zero. Deleting all such rows and columns, we obtain an orthogonal symmetric matrix $X''$ of the form
$
X'':=\npmatrix{
  D_G&F\\F\trans &-D_H
}$,
where $D_G$ 
and $D_H$ are diagonal.  Note that the diagonal entries of $D_G$  and $D_H$ are precisely the numbers $\lambda_s$ for $s=2,\dots,r-1$,  where $\mlt(\lambda_{s},D_G)=(\widetilde V \1_k)_{s-1}$, and $\mlt(\lambda_{s},D_H)=(\widetilde W \1_\ell)_{s-1}$. Since $-1<\lambda_s<1$ for $s=2,\ldots,r-1$, the matrix $F\trans F=I-D_H^2$ is invertible, and hence $F$ has trivial kernel. The identity $(X'')^2=I$ also yields $D_GF-FD_H=0$, so by Lemma~\ref{lem:Sylv} we have $\sigma(D_H)\subseteq\sigma(D_G)$. By symmetry, $F\trans $ also has trivial kernel, allowing us to conclude $\sigma(D_H)=\sigma(D_G)$. Hence, $\widetilde V\1_k=\widetilde W\1_\ell$.

It only remains to show that $\widetilde V\trans \widetilde W>0$. Let $D_{G_i}$ be the restriction of $D_G$ to the vertices in $G_i$ which survived the earlier deletion, and define $D_{H_j}$ similarly. By our earlier observations, these are all non-empty matrices.
Let $F=\left(F_{ij}\right)$ be the block partition of $F$ compatible with $\bigoplus_{i\in [k]}  D_{G_i}$ and $\bigoplus_{j\in [\ell]}  D_{H_j}$. From $(X'')^2=I$, we conclude that $D_{G_i}F_{ij}-F_{ij}D_{H_j}=0$. If $\sigma(D_{G_i})\cap \sigma(D_{H_j})=\emptyset$, then Lemma~\ref{lem:Sylv} shows that the Sylvester equation $D_{G_i}Y-YD_{H_j}=0$ has only the trivial solution $Y=0$, so $F_{ij}=0$. Tracing back, this implies that the $(i,j)$th block of $B$ is also zero, which is a contradiction. Therefore $D_{G_i}$ and $D_{H_j}$ have a common eigenvalue $\lambda_s$ for at least one $s \in \{2,\ldots,r-1\}$, for all pairs $i$ and $j$. By construction, $A_i$ and $C_j$ have the same property, so
 $(\widetilde V\trans \widetilde W)_{ij}=\sum_{s=2}^{r-1}\mlt(\lambda_s,A_i)\mlt(\lambda_s,C_j) >0$.
\end{proof}

\begin{example}
It is straightforward to see that $2K_2$ and $3 K_1$ do not have compatible multiplicity matrices, hence $q(2K_2 \vee 3K_1)\ne 2$.   
 This gives an explicit counterexample to the claim in \cite[Lemma~3.4]{MR4044603}, which was later retracted in~\cite{corrigendum}. 
 In contrast, we will see in Section~\ref{sec:complete} that $q(2K_2 \vee 2K_1)=q(2K_2 \vee 4K_1)=2$.
\end{example}

\begin{remark}
  The proof of Theorem~\ref{thm:q=2} shows that,  if $Y\in S(G\vee H)$ with $q(Y)=2$, then $Y=\alpha X+\beta I$ for some $\alpha,\beta\in\bR$, and an orthogonal symmetric matrix 
  $X \in S(G\vee H)$ where $A:=X[G]$ and $C:=-X[H]$ must have compatible multiplicity matrices.

  As an application, note that if $V,W\in \bN_0^{r\times m}$ are compatible multiplicity matrices for $G=mK_1$ and $H=mK_1$, respectively, then we must have $\widetilde V=\widetilde W=\e_i\1_m\trans$ for some $i\in\{2,\dots,r-1\}$. Hence, if $q\npmatrix{D_1&B\\B\trans&D_2}=2$ where $D_1,D_2,B$ are $m\times m$ matrices with $D_1,D_2$ diagonal and $B$ nowhere-zero, then we must have $D_1=aI_m$ and $D_2=bI_m$ for some $a,b\in \bR$.
\end{remark}

The authors have not been able to determine whether the converse to Theorem~\ref{thm:q=2} is valid in full generality. In Theorem~\ref{thm:compatible sane} below a partial converse is given, which requires additional hypotheses. For this purpose, we introduce the following terminology.

 \begin{definition}
   Let $\vv\in \bN_0^r$ be a multiplicity vector for a connected graph~$G$, and recall the notation of Definition~\ref{def:mult-vector}.
   We say that $\vv$ is \emph{sane} for $G$ (standing for \emph{\textbf{s}pectrally \textbf{a}rbitrary with \textbf{n}owhere-zero \textbf{e}igenbases}), if, for any eigenvalue list $\Lambda=(\lambda_1,\dots,\lambda_r)\in \bR^r$ with $\lambda_i < \lambda_{i+1}$ for $i\in [r-1]$, there is a nowhere-zero orthogonal matrix $U$ so that $UD_{\Lambda,\vv} U\trans\in S(G)$. If, for any finite set $\mc Y\subseteq \bR^n\setminus\{0\}$  and any $\Lambda$, an orthogonal matrix $U$ can be chosen as above with the additional property that $U\yy$ is nowhere-zero for all $\yy \in \mc Y$, then we say that $\vv$ is \emph{generically realisable} for~$G$.
\end{definition}

\begin{definition}
  Let $V$ be a multiplicity matrix for a graph $G$. If every
  column of $V$ is sane/generically realisable for the
  corresponding connected component of $G$, then we say that $V$ is \emph{sane/generically realisable} for $G$.
\end{definition}

\begin{definition}
  If every multiplicity matrix for a graph~$G$ is sane for~$G$, then we say that $G$ is \emph{sane}.
  If every multiplicity matrix for a graph~$G$ is generically realisable for~$G$, then we say that $G$ is \emph{generically realisable}.
\end{definition}

\begin{remark}\label{0-1wb}
  Clearly, for a given graph~$G$, any generically realisable multiplicity vector is sane, and any sane multiplicity vector is spectrally arbitrary. In fact, in some situations those concepts agree, but as we will see in Section~\ref{sec:examples}, there are cases where they are different.

Sane multiplicity vectors were considered in the literature before, although not under that name. For example,  it is well known that if $\vv\in \bN_0^r$ and $n:=\1_r\trans\vv\ge2$, then $\vv$ is spectrally arbitrary for $K_n$ if it has at least two positive coordinates, and that any such $\vv$ is sane for $K_n$ (see, e.g.,~\cite[Theorem~4.5]{MR3208410}). Furthermore, any multiplicity vector $\vv\in \{0,1\}^{r}$ (i.e., with all its elements equal to $0$ or $1$) is sane for every connected graph $G$ of order $\1_r\trans \vv$, by~\cite[Theorem~4.2]{MR3034535}.

In fact, $K_n$ turns out to be generically realisable, as we will see in Proposition~\ref{complete-generic} below. In~\cite{LOS21}, we will prove that $P_n$ is generically realisable as well.
\end{remark}

 To prove Theorem~\ref{thm:compatible sane} we will need the following technical lemma. For $q\in \bN$, we write $SO(q)$ for the set of $q\times q$ special orthogonal matrices, i.e.,~all $q \times q$ orthogonal matrices with determinant~$1$, considered as an algebraic subvariety of $\bR^{q\times q}$.  For $\p=(p_1,\dots,p_r)\trans \in \bN^r$, we write
 $ SO(\p):=SO(p_1)\times \dots\times SO(p_r)$.
We will use the elementary fact from algebraic geometry that the algebraic variety $SO(q)$ is irreducible for any $q\in \bN$ (see, e.g.,~\cite{MR3331229}). Since $SO(\p)$ is a product of such varieties, it is also an irreducible algebraic variety.

 \begin{lemma}
   \label{lem:orthogonal3}
   Let $r\in \bN$, $\p=(p_1,\dots, p_r)\trans \in \bN^r$ and for $s \in [r]$, let $A_s\in  \bR^{p_s\times p_s}$ with $\rk A_s=1$. If $\sum_{s=1}^r\tr(A_sX_s)=0$ for all $(X_1,\dots,X_r)\in SO(\p)$, then $r \geq 2$ and $p_s=1$ for each $s\in [r]$.
 \end{lemma}

 \begin{proof}
   For $s\in [r]$, write $A_s=a_s\vv_s\ww_s\trans $ where $\vv_s,\ww_s\in\bR^{p_s}$ are unit vectors and $a_s\in\bR\setminus\{0\}$. Suppose for a contradiction that $p_s>1$ for some $s\in[r]$. Then $SO(p_s)$  acts transitively on the unit vectors in $\bR^{p_s}$, so there exist $X_s,X_s'\in SO(p_s)$ with $X_s\vv_s= \ww_s=-X_s'\vv_s$. For $t\in [r]\setminus\{s\}$, choose any $X_t\in SO(p_t)$ and set $X_t'=X_t$. We have $\sum_{t=1}^r\tr(A_tX_t)=0=\sum_{t=1}^r(A_tX_t')$, and cancelling common terms gives $\tr(A_sX_s)=\tr(A_sX_s')$. This implies that $a_s=0$, a contradiction. We conclude that $p_s=1$ for each $s\in [r]$. In particular, $SO(p_1)=\{1\}$, hence $\tr(A_11)=\pm a_1\ne0$, so $r\ne1$.
 \end{proof}

In the next result we identify a particular situation that Lemma \ref{lem:orthogonal3} will eventually be applied to.

 \begin{lemma}\label{lem:rankones}
  Let $k,m,n \in \bN$. For $s\in [k]$, let $p_s \in \bN$ and $\emptyset\ne \mc R_s,\mc C_s\subseteq [p_s]$ where $m':=\sum_{s\in [k]}|\mc R_s|\le m$ and $n':=\sum_{s\in [k]}|\mc C_s|\le n$.
  Let 
  $F: \times_{s=1}^k \bR^{p_s\times p_s}\to \bR^{n\times m}$ be given by
   \[F(X_1,\dots,X_k):=\npmatrix{\bigoplus_{s\in [k]}X_s[\mc R_s,\mc C_s]&0_{m'\times (n-n')}\\0_{(m-m')\times n'}&0_{(m-m')\times (n-n')}}.\]
   Fix $a\in [m]$, $b\in [n]$ and invertible nowhere-zero matrices $S\in \bR^{n\times n}$ and $T\in \bR^{m\times m}$. Then there exist matrices $A_s\in \bR^{p_s\times p_s}$ with $\rk A_s=1$ for $s\in [k]$ so that
$ \left(S\trans F(X_1,\dots,X_k)T\right)_{ab}=\sum_{s\in [k]}\tr(A_sX_s)$
   for any $X_1\in \bR^{p_1\times p_1},\dots,X_k\in \bR^{p_k\times p_k}$.
 \end{lemma}


 \begin{proof}
   Let us write $F(\mc X)=F(X_1,\dots,X_k)$.
   Define $\mc E_s:=(\sum_{1\le t<s}|\mc R_t|) +[|\mc R_s|]\subseteq [m]$ and $\mc F_s:=(\sum_{1\le t<s}|\mc C_t|)+[|\mc C_s|]\subseteq [n]$, so that for $s\in [k]$ we have
$X_s[\mc R_s,\mc C_s]=F(\mc X)[\mc E_s,\mc F_s]$. Let $\chi_{\mc E_s\times \mc F_s}\in \bR^{m\times n}$ denote the matrix with ones on $\mc E_s\times \mc F_s$ and zeros everywhere else, and let  $\circ$ denote the Hadamard product. The block structure in the definition of $F(\mc X)$ implies that
$F(\mc X)=F(\mc X)\circ \sum_{s\in [k]}\chi_{\mc E_s\times \mc F_s}$.
   Observe that the matrix $Y:=(T\e_a)(S\e_b)\trans\in \bR^{n\times m}$ is rank one and nowhere-zero, so any 
   submatrix of $Y$ shares these properties. Define $A_s\in \bR^{p_s\times p_s}$ by setting $A_s[\mc C_s,\mc R_s]:=Y[\mc F_s,\mc E_s]$ and defining all other entries to be zero (i.e., $(A_s)_{\alpha,\beta}=0$ if $(\alpha,\beta)\not\in \mc C_s\times \mc R_s$). Then $\rk A_s=1$, and we have
   $  \tr(A_sX_s)=\tr(A_s[\mc C_s,\mc R_s]X_s[\mc R_s,\mc C_s])
                =\tr(Y[\mc F_s,\mc E_s]F(\mc X)[\mc E_s,\mc F_s])
                =\tr(Y(F(\mc X)\circ \chi_{\mc E_s\times \mc F_s}))$,
   so
   \begin{align*}
     \sum_{s\in [k]}\tr(A_sX_s)&=\tr\left(Y\Big(F(\mc X)\circ \sum_{s\in[k]}\chi_{\mc E_s\times \mc F_s}\Big)\right)
                               =\tr(YF(\mc X)) \\
                               &= \tr(T\e_a(S\e_b)\trans F(\mc X))
                               =\e_b\trans S\trans F(\mc X)T\e_a = (S\trans F(\mc X)T)_{ab}.\qedhere
   \end{align*}
 \end{proof}

\begin{theorem}\label{thm:compatible sane}
  Suppose that $G$ and $H$ have compatible sane multiplicity matrices $V=(v_{si})\in \bN_0^{r\times k}$ and $W=(w_{sj})\in \bN_0^{r\times \ell}$. Let $\p=(p_2,\dots,p_{r-1})\trans :=\widetilde V\1_k=\widetilde W\1_\ell$.
Then $q(G\vee H)\le 2$ if at least one of the following conditions is satisfied:
\begin{enumerate}
\item\label{c1} Whenever $(\widetilde V\trans \widetilde W)_{ij} \geq 2$, there exists $s \in \{2,\ldots,r-1\}$ with $ v_{s i} w_{s j} \neq 0$ and $p_s \geq 2$. 
\item\label{c2} $V$ is a generically realisable multiplicity matrix for $G$.
\item\label{c3} $W$ is a generically realisable multiplicity matrix for $H$.
\end{enumerate}
\end{theorem}

\begin{proof}
Let $G_1,\dots,G_k$ and $H_1,\dots,H_\ell$ be the connected components of $G$ and
$H$, 
%
and let $\Lambda=(-1,\lambda_2,\dots,\lambda_{r-1},1)$ where $-1<\lambda_2<\ldots< \lambda_{r-1}<1$.
%
Write $\vv_i$ for $i$th column of $V$ and let $D_G\in \bR^{|G|\times |G|}$ be the diagonal matrix $D_G=\bigoplus_{i \in [k]} D_{G_i}$, where
$D_{G_i}=D_{\Lambda,\vv_i}$.
Similarly, let $\ww_j$ be the $j$th column of $W$ and let  $D_H\in \bR^{|H|\times |H|}$ be the diagonal matrix $D_H=\bigoplus_{j=1}^\ell D_{H_j}$, where
$D_{H_j}=D_{\Lambda,\ww_j}$.
Note that for $s\in \{2,\ldots,r-1\}$, we have $p_s=\mlt(\lambda_s,D_G)=\mlt(\lambda_s,D_H)$.

%

Let $\mc Z=(Z_1,\dots,Z_r)\in SO(\p)$. The freedom in the choice of $\mc Z$ will be needed at a later stage.  
For each $s\in \{2,\ldots,r-1\}$, let $U_s\in \bR^{2p_s\times 2p_s}$ be the orthogonal symmetric matrix
$$
U_s:=\npmatrix{
  \lambda_sI_{p_s}&(1-\lambda_s^2)^{1/2}Z_s\\(1-\lambda_s^2)^{1/2}Z_s\trans &-\lambda_s I_{p_s}
}.$$
Note that $U_s$ depends on our choice of $\mc Z$.
Set $n_{1}:=\e_1\trans V\1_k+\e_{r}\trans W\1_\ell$ to be the sum of the entries in the first row of $V$ and the last row of $W$, and $n_{r}:=\e_{r}\trans V\1_k+\e_{1}\trans W\1_\ell$ to be the sum of the entries in the first row of $W$ and the last row of $V$. Consider the orthogonal matrix
 $ U_0:=\left(\bigoplus_{s=2}^{r-1} U_s\right)\oplus  \left(-I_{n_{1}}\right)\oplus I_{n_{r}}$.
By construction, the diagonal of $U_0$ is a permutation of the diagonal of $D_G\oplus (-D_H)$. It follows that $U_0$ is permutation similar to an orthogonal matrix
$U=  \npmatrix{   D_G&B(\mc Z)\\B(\mc Z)\trans &-D_H
  }$
under a permutation which maps the top left block of each $U_s$ to a submatrix of $D_G$, and the bottom right block to a submatrix of $-D_H$, where $B(\mc Z)$ is a $|G|\times |H|$ matrix which depends on $\mc Z$.  
We partition $B(\mc Z)$  as a $k\times \ell$ block matrix $B(\mc Z)=(B_{ij}(\mc Z))_{i\in [k],j\in [\ell]}$  with block-partition compatible with those of $D_G=\bigoplus_{i=1}^k D_{G_i}$ and $D_H=\bigoplus_{j=1}^\ell D_{H_j}$, so that each $B_{ij}(\mc Z)$ is an $|G_i| \times |H_j|$ matrix. For $(i,j)\in [k]\times [\ell]$, let $Q(i,j)=\{s\colon 2\leq s \leq r-1,  v_{si} w_{sj}\ne 0\}$. Since $V$ and $W$ are compatible, $(\widetilde V\trans \widetilde W)_{ij}\ne 0$ and so $Q(i,j)\ne \emptyset$ for every~$i,j$.
For $2\le s\le r-1$, note that $\lambda_s$ is a diagonal value of both $D_{G_i}$ and $D_{H_j}$ if and only if $s\in Q(i,j)$. By construction, the rows and columns of $B_{ij}(\mc Z)$ may be permuted to obtain the $|G_i|\times|H_j|$  matrix \begin{equation}\label{eq:Bprime}B_{ij}'(\mc Z):=\npmatrix{\bigoplus_{s \in Q(i,j)} (1-\lambda_s^2)^{1/2} Z_s[\mc R_{si},\mc C_{sj}]& 0\\0&0}
\end{equation}
where for each $s\in Q(i,j)$ both $\{\mc R_{si}\colon i\in [k]\}$ and $\{\mc C_{sj}\colon j\in [\ell]\}$ are partitions of $[p_s]$, with $|\mc R_{si}|=\mlt(\lambda_s,D_{G_i})=v_{si}$ and $|\mc C_{sj}|=\mlt(-\lambda_s,-D_{H_j})=w_{sj}$,  and where the zeros in the matrix $B_{ij}'(\mc Z)$ above represent the (possibly empty) zero matrices which pad the direct sum out to $|G_i|$ rows and $|H_j|$ columns.

Since $V$ and $W$ are sane multiplicity matrices for~$G$ and~$H$, respectively, there exist nowhere-zero orthogonal matrices $S_i,T_j$ for $i\in [k]$, $j\in [\ell]$, such that $S_iD_{G_i}S_i\trans \in S(G_i)$ and
$T_jD_{H_j}T_j\trans \in S(H_j)$. Let  $\mc S:=(S_1,\ldots,S_k)$ and $\mc T:=(T_1,\ldots,T_\ell)$, and note that $\mc S$ and $\mc T$ can typically be chosen in several different ways. 
The matrix $R:=(\bigoplus_{i=1}^k S_i) \oplus (\bigoplus_{j=1}^\ell T_j)$ is orthogonal and thus
$R U R\trans =\npmatrix{A_G & C(\mc Z) \\ C(\mc Z)\trans  & A_H}$ is an orthogonal symmetric matrix with $A_G \in S(G)$ and $A_H \in S(H)$. Hence, provided the matrix $C(\mc Z)=(\bigoplus_{i=1}^k S_i)B(\mc Z)(\bigoplus_{j=1}^\ell T_j)\trans $ has no zero elements we will have  $RUR\trans  \in S(G \vee H)$ and so $q(G\vee H)\le 2$, as required. To analyse when this happens, we will now refer back to the freedom we have in choosing $\mc Z$, $\mc S$ and $\mc T$.

Note that the matrix $C(\mc Z)$ inherits the block-partition of $B(\mc Z)$, i.e., $C(\mc Z)$ is a $k \times \ell$ block matrix with $|G_i|\times |H_j|$ blocks $C_{ij}(\mc Z)=S_i B_{ij}(\mc Z) T_j\trans $. 
We will first fix $\mc S$ and $\mc T$ arbitrarily and show that hypothesis~\ref{c1} of the theorem guarantees that $C(\mc Z)$ 
is nowhere-zero for some appropriate choice of $\mc Z$.

\newcommand{\bb}{B_{i_0,j_0}(\mc Z)}%
\newcommand{\bbp}{B_{i_0,j_0}'(\mc Z)}%
For $i\in [k]$, $j\in [\ell]$, $a\in [|G_i|]$ and $b\in [|H_j|]$, consider the the linear functionals
 $L_{ab}(i,j):SO(\p)\to \bR $ given by
 $L_{ab}(i,j)(\mc Z):=(C_{ij}(\mc Z))_{ab}$.
If $q(G\vee H)>2$, then $C(\mc Z)$ has at least one entry equal to zero for any choice of $\mc Z \in SO(\p)$, so $SO(\p)\subseteq \bigcup_{i,j,a,b}L_{ab}(i,j)^{-1}(0)$.
Since $SO(\p)$ is an irreducible algebraic variety, there exist $i_0,j_0,a_0,b_0$ so that for $L_0:=L_{a_0,b_0}(i_0,j_0)$ we have $SO(\p) \subseteq L_{a_0,b_0}(i_0,j_0)^{-1}(0)$, i.e., $L_0(\mc Z)=0$ for $\mc Z \in SO(\p)$.
On the other hand, $L_0(\mc Z)=(S_{i_0}\Pi_1\bbp\Pi_2T_{j_0}\trans )_{a_0b_0}$ where $\Pi_1,\Pi_2$ are permutation matrices with $\Pi_1\bbp\Pi_2=\bb$. By equation~(\ref{eq:Bprime}) and Lemma~\ref{lem:rankones}, there exist matrices $A_s\in \bR^{p_s\times p_s}$ with $\rk A_s=1$ for $s\in Q(i_0,j_0)$   so that $0=L_0(\mc Z)=\sum_{s\in Q(i_0,j_0)}\tr(A_sZ_s)$ for every $\mc Z\in SO(\p)$. By Lemma~\ref{lem:orthogonal3}, we have $(\widetilde V\trans\widetilde W)_{i_0j_0}\ge |Q(i_0,j_0)|\geq 2$, and $p_s=1$ for every $s\in Q(i_0,j_0)$. On the other hand, for $s\in \{2,\dots,r-1\}\setminus Q(i_0,j_0)$, we have $(\widetilde V\trans\widetilde W)_{i_0j_0}=0$.
This is inconsistent with hypothesis~\ref{c1}. Hence, when~\ref{c1} holds, we can find a suitable $\mc Z$ for which the orthogonal matrix $RUR\trans$ lies in $S(G\vee H)$, so $q(G\vee H)=2$ in this case.

Now suppose that hypothesis~\ref{c2} holds. Fix any 
  invertible nowhere-zero $|G_i|\times |G_i|$ matrices~$\widetilde S_i$ for $i\in [k]$, and
  consider the linear functionals
 $\widetilde L_{ab}(i,j):SO(\p)\to \bR$ given by $\widetilde L_{ab}(i,j)(\mc Z):=(\widetilde S_i B_{ij}(\mc Z)T_j\trans )_{ab}$.
Suppose that for any $\mc Z\in SO(\p)$, there exist $i,j,b$ so that the $b$th column of $B_{ij}(\mc Z)T_j\trans $ is zero. Then $SO(\p)\subseteq \bigcup_{i,j,b}\bigcap_a \widetilde L_{ab}(i,j)^{-1}(0)$,
and by the irreducibility of the algebraic variety $SO(\p)$ there exist some fixed $i_0,j_0,b_0$ so that $SO(\p)\subseteq \bigcap_a \widetilde L_{ab_0}(i_0,j_0)^{-1}(0)$. In other words, for every $\mc Z\in SO(\p)$, the $b_0$th column of $\widetilde S_{i_0} B_{i_0j_0}(\mc Z)T_{j_0}\trans $ is zero, so by the invertibility of $\widetilde S_{i_0}$, the $b_0$th column of $B_{i_0j_0}(\mc Z)T_{j_0}\trans $ is zero. Now (taking $a=1$) observe that $\widetilde L:=\widetilde L_{1b_0}(i_0,j_0)$ vanishes at every $\mc Z\in SO(\p)$. Since $\widetilde S_{i_0}$ is nowhere-zero, Lemma~\ref{lem:rankones} implies that $\widetilde L$ may be written in the form $\widetilde L(\mc Z)=\sum_{s\in Q}\tr(A_sZ_s)$, where $\emptyset\ne Q:=Q(i_0,j_0)$ and $\rk A_s=1$ for each $s\in Q$. By Lemma~\ref{lem:orthogonal3}, for each $s\in Q$ we have $p_s=1$. So for every $\mc Z\in SO(\p)$ and $s\in Q$, we have $Z_s\in SO(1)$, i.e., $Z_s=1$. This implies that $B_{i_0j_0}'(\mc Z)=\npmatrix{D& 0\\0&0}$ where $D=\diag((1-\lambda_s^2)^{1/2}:s\in Q)$ is an invertible diagonal matrix and the $0$s represent (possibly empty) zero matrices. (In particular, $B'_{i_0j_0}(\mc Z)$ is independent of $\mc Z$.) This implies that the kernel of $B'_{i_0j_0}(\mc Z)$ does not contain a nowhere-zero vector, and the same is therefore true of the permutation equivalent matrix $B_{i_0j_0}(\mc Z)$. Since $T_{j_0}$ is nowhere-zero, it follows that the $b_0$th column of $B_{i_0j_0}(\mc Z)T_{j_0}\trans $ is not zero, a contradiction.

Therefore it is possible to choose $\mc Z\in SO(\p)$ so that $B_{ij}(\mc Z) T_j\trans $ has no column equal to zero for all pairs $i$ and $j$. Since $V$ is generically realisable for $G$, we can find a $k$-tuple  $\mc S=(S_1,\dots,S_k)$ of nowhere-zero orthogonal matrices such that for every $i\in [k]$, we have $S_i D_{G_i}S_i\trans \in S(G_i)$ and for every $j\in [\ell]$, the matrix $C_{ij}(\mc Z)=S_i B_{ij}(\mc Z) T_j\trans $ is nowhere-zero. Hence, $q(G\vee H)\le 2$ in this case. The argument for hypothesis~\ref{c3} is of course symmetric.  
\end{proof}

Theorems~\ref{thm:q=2} and~\ref{thm:compatible sane} yield:
\begin{corollary}\label{compat-iff}
  If $G$ is generically realisable and $H$ is sane,
  then $q(G\vee H)= 2$ if and only if $G$ and $H$ have a pair of compatible multiplicity matrices.
\end{corollary}

\section{Remarks, applications, examples}\label{sec:examples}


Our principal application of Theorem~\ref{thm:compatible sane} in this paper is to determine the minimum number of distinct eigenvalues of the join of two graphs, each of which is a union of complete graphs. Before turning to that in Section~\ref{sec:complete}, we give some other consequences and related discussion.

\subsection{Construction of orthogonal symmetric matrices}
 As is apparent from the proof of Theorem \ref{thm:compatible sane}, several things would have to align for two graphs $G$ and $H$ with compatible multiplicity matrices to have $q(G \vee H) > 2$, and this won't happen in any sufficiently generic situation. In other words, once two graphs have compatible multiplicity matrices, they will typically have $q(G \vee H)$ equal to $2$. 
 This potentially gives us a wealth of examples, including situations in which we cannot verify the technical hypotheses of Theorem~\ref{thm:compatible sane}.

 In Pseudo-algorithm~\ref{alg:pseudo}, we concisely summarise the construction used repeatedly in the previous section. This is illustrated by the example involving cycles below.

\begin{algorithm}
  Suppose $G$ and $H$ are graphs with $k$ and $\ell$ components, respectively, and
 $A_G\in S(G)$ and $A_H\in S(H)$ have compatible multiplicity matrices $V\in \bN_0^{r\times k}$ and $W\in \bN_0^{r\times \ell}$, with respect to $\Lambda=(-1,\lambda_2,\dots,\lambda_{r-1},1)$ where $\lambda_i\in (-1,1)$ for $i=2,\dots,r-1$.
  \begin{enumerate}[itemsep=1ex,topsep=1ex]\renewcommand{\labelenumi}{Step~\arabic{enumi}.}
  \item Choose orthogonal matrices $S,T$ with $S\trans A_G S=D_{\Lambda,V\1_k}$ and $T\trans A_H T = D_{\Lambda,W\1_\ell}$.
    Choose $Z_i\in O(p_i)$ for $i=2,\ldots,r-1$, where $p_i=\e_i\trans V\1_k=\e_i\trans W\1_\ell$.
\item Let $C:=S\left(0_{\e_1\trans V\1_k\times \e_1\trans W\1_\ell}\oplus \left(\bigoplus_{s=2}^{r-1} (1-\lambda_s^2)^{1/2}Z_s\right)\oplus 0_{\e_r\trans V\1_k\times \e_r\trans W\1_\ell}\right)T\trans$.
\item If $C$ is nowhere-zero, return 
  $X:=\npmatrix{A_G & C\\C\trans & -A_H}$.
Otherwise, go to Step~1, making a different choice of orthogonal matrices $S$, $T$ and $Z_i$, if possible.
\end{enumerate}
Of course, we have no guarantee that this procedure will terminate; this is why the technical hypotheses of Theorem~\ref{thm:compatible sane} were imposed.
\caption{
  Finding an orthogonal symmetric matrix in $S(G\vee H)$}\label{alg:pseudo}
\end{algorithm}


\begin{example}\label{eg:cycles}
Let $G=C_8$ and $H=C_4$.
To construct an orthogonal symmetric matrix $X \in S(G \vee H)$, let $\vv=(2,2,2,2)\trans$ and $\ww=(0,2,2,0)\trans$. These are compatible multiplicity vectors  for $G$ and $H$, respectively; for example, 
\[ A_G:=\tfrac1{\sqrt{2+\sqrt2}}
    \left(
      \begin{smallmatrix}
        0 & 1 & 0 & 0 & 0 & 0 & 0 & -1 \\
        1 & 0 & 1 & 0 & 0 & 0 & 0 & 0 \\
        0 & 1 & 0 & 1 & 0 & 0 & 0 & 0 \\
        0 & 0 & 1 & 0 & 1 & 0 & 0 & 0 \\
        0 & 0 & 0 & 1 & 0 & 1 & 0 & 0 \\
        0 & 0 & 0 & 0 & 1 & 0 & 1 & 0 \\
        0 & 0 & 0 & 0 & 0 & 1 & 0 & 1 \\
        -1 & 0 & 0 & 0 & 0 & 0 & 1 & 0 \\
      \end{smallmatrix}
    \right)\in S(G),\quad
    A_H:=\tfrac1{2+\sqrt2}\left(
\begin{smallmatrix}
 0 & 1 & 0 & -1 \\
 1 & 0 & 1 & 0 \\
 0 & 1 & 0 & 1 \\
 -1 & 0 & 1 & 0 \\
\end{smallmatrix}
\right)\in S(H)\]
have these multiplicities with respect to the eigenvalue list $\Lambda=(-1,-\lambda,\lambda,1)$, where $\lambda=\sqrt{2}-1$.
We 
execute Pseudo-algorithm~\ref{alg:pseudo} using a computer algebra system. With a little trial and error, it turns out that we can do this in exact arithmetic with $Z_2=I_2$ and $Z_3=\frac15\left(\begin{smallmatrix}4&-3\\3&4\end{smallmatrix}\right)$, to obtain the orthogonal symmetric matrix $X=\npmatrix{A_G&C\\C\trans &-A_H}\in S(G\vee H)$, where for $\alpha(r):=\sqrt{5+\tfrac{r}{\sqrt{2}}}$, we have 
\[C:=\frac{\sqrt\lambda}{10}
\left(
\begin{smallmatrix}
  3 \sqrt{2} & -1            & 6 \sqrt{2}  & 3             \\[2pt]
 -\alpha(1)  & 3 \alpha(-7)  & -\alpha(-1) & 3 \alpha(7)   \\[2pt]
 -9          & -\sqrt{2}     & -3          & -2 \sqrt{2}   \\[2pt]
 \alpha(7)   & -3 \alpha(1)  & -\alpha(-7) & -3 \alpha(-1) \\[2pt]
 6 \sqrt{2}  & 3             & -3 \sqrt{2} & 1             \\[2pt]
 -\alpha(-1) & 3 \alpha(7)   & \alpha(1)   & -3 \alpha(-7) \\[2pt]
 -3          & -2 \sqrt{2}   & 9           & \sqrt{2}      \\[2pt]
 -\alpha(-7) & -3 \alpha(-1) & -\alpha(7)  & 3 \alpha(1)   \\
\end{smallmatrix}
\right)\approx
\left(
\begin{smallmatrix}
\+0.27 & -0.064 &\+0.55 &\+0.19 \\[2pt]
 -0.15 &\+0.043 & -0.13 &\+0.61 \\[2pt]
 -0.58 & -0.091 & -0.19 & -0.18 \\[2pt]
\+0.20 & -0.46 & -0.014 & -0.40 \\[2pt]
\+0.55 &\+0.19 & -0.27 &\+0.064 \\[2pt]
 -0.13 &\+0.61 &\+0.15 & -0.043 \\[2pt]
 -0.19 & -0.18 &\+0.58 &\+0.091 \\[2pt]
 -0.014 & -0.40 & -0.20 &\+0.46 \\
\end{smallmatrix}
\right).\]
%
\end{example}

\subsection{Some families of graphs realisable by an orthogonal symmetric matrix}
In \cite[Theorem~5.2]{MR3506498} the authors proved that $q(G\vee H)=2$ if $G$ and $H$ are connected graphs with $|G|=|H|$. We can use Theorem \ref{thm:compatible sane}\ref{c1} to extend their result to more than one connected component. (In fact, the statement below is also valid for $k=1$~\cite{LOS21}.)
\begin{theorem}
  If $G$ and $H$ are graphs each having $k\ge 2$ connected components, and there is $n\in \bN$ so that the order of every one of these connected components is either $n$, $n+1$ or $n+2$,
  then 
  $q(G\vee H)=2$.
\end{theorem}
\begin{proof}
  Let $G_1,\dots,G_k$ and $H_1,\dots,H_k$ be the connected components of $G$ and $H$, respectively. Consider the matrices
 $$V:=\npmatrix{{\vv}\trans\\E\\{\vv'}\trans}\in \bN_0^{(n+2)\times k}\text{ and }
 W:=\npmatrix{{\ww}\trans\\E\\{\ww'}\trans}\in \bN_0^{(n+2)\times k},$$
 where $E:=\1_n\1_k\trans\in \R^{n \times k}$ is the matrix full of ones and $\vv=(v_i)_{i\in [k]}$, $\vv'=(v'_i)_{i\in [k]}$, $\ww=(w_i)_{i\in [k]}$ and $\ww'=(w'_i)_{i\in [k]}$ are all $0$-$1$ vectors in $\{0,1\}^{k}$, such that $v_i+v'_i=|G_i|-n$ and $w_i+w'_i=|H_i|-n$ for $i\in [k]$.
 Note that $V$ and $W$ are  sane  multiplicity matrices for $G$ and $H$, respectively, by Remark \ref{0-1wb}. Moreover, $V$ and $W$ are compatible since $\widetilde V=\widetilde W$ 
 and $\widetilde V\trans \widetilde W =n \1_k\1_k\trans>0$.
 Since $k\geq 2$, Theorem \ref{thm:compatible sane}\ref{c1} implies that $q(G \vee H)=2$.
\end{proof}

%
%

It is easy to see that a complete graph is generically realisable (see Proposition~\ref{complete-generic} below for details). This fact together with Theorem \ref{thm:compatible sane}\ref{c3} allows us to extend
 \cite[Lemmas 3.13 and 3.14]{MR4044603}, where the authors proved that if $G$ is a connected graph of order $\ell$ or $\ell+1$, then $q(G\vee \ell K_1)=2$.

\begin{theorem}\label{cor:G Kn}
 If $G$ is a connected graph of order $m \in \{\ell,\ell+1,\ell+2\}$  and $n_1,\ldots,n_\ell\in \bN$, 
 then $q(G \vee \bigcup_{j\in [\ell]} K_{n_j})=2$.
 \end{theorem}

 \begin{proof}
Let us define
 $$V:=\npmatrix{\varepsilon \\ \1_{\ell}\\ \varepsilon'}\in \bN_0^{\ell+2} \text{ and } W:=\npmatrix{{\n\trans}\\
I_{\ell}\\
{{\bf 0}_\ell}\trans} \in \bN_0^{(\ell+2) \times \ell},$$
where $\varepsilon',\varepsilon''\in \{0,1\}$ and ${\n}=\npmatrix{n_1-1 & \ldots & n_\ell-1}\trans\in \bN_0^{\ell}$. By Remark \ref{0-1wb}, $V$ is a sane multiplicity matrix for $G$, and by Proposition \ref{complete-generic}, $W$ is a generically realisable multiplicity matrix for $\bigcup_{j\in [\ell]} K_{n_j}$.
 Since $\widetilde V =\widetilde W \1_{\ell}=\1_{\ell}$ and $\widetilde V\trans \widetilde W =\1_{\ell}\trans>0$, $V$ and $W$ are compatible. It follows by Theorem \ref{thm:compatible sane}\ref{c3} that $q(G \vee \bigcup_{j\in [\ell]} K_{n_j})=2$.
 \end{proof}
%
%
%



\subsection{Generic realisability and sanity}

In the remainder of this section, we explore the relationship between the sane and the generically realisable multiplicity vectors. First, we show that provided its minimum multiplicity is greater than one, any sane multiplicity vector is automatically generically realisable.

\begin{proposition}\label{prop:gen}
If $\vv \in  \bN_0^r$ is a sane multiplicity vector for $G$ with no entry equal to $1$, then $\vv$ is also generically realisable for $G$.
\end{proposition}

\begin{proof}Generic realisability for $G$ is not changed by inserting or deleting zeros from $\vv=\npmatrix{v_1 & v_2&\ldots &v_{r} }\trans$, so we may assume that  $v_i\ge 2$ for all $i\in [r]$.
Let $\Lambda=(\lambda_1,\dots,\lambda_r)\in \bR^r$ with $\lambda_i < \lambda_{i+1}$ for $i\in [r-1]$. Since $\vv$ is sane for $G$, there is a nowhere-zero orthogonal matrix $U\in \R^{n \times n}$ where $n=|G|$ with $UD_{\Lambda,\vv} U\trans \in S(G)$.
Note that for $\mc X=(X_1,\dots,X_r)\in SO(\vv)$, we have $(\oplus_{i=1}^r X_i)D_{\Lambda,\vv} (\oplus_{i=1}^r X_i\trans)=D_{\Lambda,\vv}$, hence $U$ can be replaced by $U(\oplus_{i=1}^r X_i)$.

Let $\mc Y \subset \R^n\setminus\{0\}$ be finite. For every $\yy \in \mc Y$, and $j \in [n]$, we define linear functionals:
 \begin{align*}
 &L(j,\yy):SO(\vv)\to \bR \\
 &L(j,\yy)(\mc X):=(U(\oplus_{i=1}^r X_i)\yy )_{j}.
 \end{align*}
 If $\vv$ is not generically realisable, then there exists $U$ as above, and some finite set of nonzero vectors $\mc Y$ so that:
  \[SO(\vv)\subseteq \bigcup_{j \in [n], \yy \in \mc Y}L(j,\yy)^{-1}(0).\]
Irreducibility of the algebraic variety $SO(\vv)$ guarantees the existence of a  fixed $j_0 \in [n]$ and $\yy_0 \in \mc Y$ so that $SO(\vv)\subseteq L(j_0,\yy_0)^{-1}(0)$. For $\mc X=(X_1,\dots,X_r)\in SO(\vv)$, we can write:
\[L(j_0,\yy_0)(\mc X)=\tr \left((\yy_0\e_{j_0}\trans U) (\oplus_{i=1}^r X_i)\right).\]
Let $R_i:=[0_{v_i\times (v_1+\dots+v_i-1)}\;I_{v_i}\; 0_{v_i\times (v_{i+1}+\dots+v_r)}]$, $\xx_i:=R_i\yy_0$, $\uu_i:=R_iU\trans \e_{j_0}$, and $A_i:=\xx_i\uu_i\trans$. 
Since $U$ is nowhere-zero, we have $\uu_i\ne 0$ for every $i\in [r]$, and $\yy_0\ne 0$, so the set $Q=\{i\in [r]\colon \xx_i\ne 0\}$ is non-empty, and $L(j_0,\yy_0)(\mc X)=\sum_{i\in Q}\tr(A_iX_i)$ with $\rk A_i=1$ for all $i\in Q$. Since $v_i\ge 2$ for all $i$, this contradicts Lemma~\ref{lem:orthogonal3}.
\end{proof}

Our next task is to find an example of a sane multiplicity vector which is not generically realisable. First, a simple lemma.

\begin{lemma}\label{lem:bases}
  If ${\mathcal V}$ is a subspace of $\bR^n$ which contains a nowhere-zero vector,
  then there is a nowhere-zero orthonormal basis for ${\mathcal V}$.
\end{lemma}
\begin{proof}
  The case $\dim {\mathcal V}=1$ is trivial. Assume inductively that $r=\dim {\mathcal V} \ge2$ and $\bbb_1,\dots,\bbb_r$  is an orthonormal basis of ${\mathcal V}$ for which $\bbb_1,\dots,\bbb_{r-1}$ are nowhere-zero. For $0<t<1$, replace $\bbb_{r-1}$ with $\bbb_{r-1}'=\sqrt{1-t^2}\bbb_{r-1}+ t\bbb_r$ and replace $\bbb_r$ with $\bbb_r'=t\bbb_{r-1}+\sqrt{1-t^2}\bbb_r$. Since $\bbb_{r-1}$ is nowhere-zero, for sufficiently small $t>0$ the vectors $\bbb_{r-1}'$ and $\bbb_r'$ are nowhere-zero.
\end{proof}

\begin{example}
We claim that $\vv=\npmatrix{1 & m & 1}\trans  \in \bN_0^{3}$, $m \geq 3$,  is sane, but not generically realisable for $K_2 \vee mK_1$, showing that the two concepts are indeed different.

In fact, $\vv$ is not generically realisable for $K_2\vee mK_1$ with respect to any list $\Lambda$ of three distinct eigenvalues. To see this, it suffices by translation and scaling to take $\Lambda=(-1,0,\lambda)$ where $\lambda>0$. 
So, let $X=\npmatrix{A & B\\ B\trans & D}\in S(K_2\vee mK_1)$
be a rank two matrix with simple eigenvalues $-1$ and $\lambda$, where $A\in S(K_2)$ and $D$ is diagonal. Since $\rank X=2$ and $m>2$, at least one diagonal entry of $D$ is zero. Further, $B$ is nowhere-zero and $\rank X=2$, and this in turn implies that $D=0$.  Considering the spectral decomposition $X=\lambda \xx\xx\trans- \ww\ww\trans$ where $\xx,\ww$ are orthogonal unit eigenvectors with eigenvalues $\lambda,-1$, we write
$$\xx=\npmatrix{\sqrt{1-s^2}\xx_1\\s\xx_2} \text{ and }\ww=\npmatrix{\sqrt{1-t^2}\ww_1\\t\ww_2}$$ where $0\le s,t\le 1$ and $\xx_i,\ww_i$ are unit vectors with $\xx_1,\ww_1\in \bR^2$ and $\xx_2,\ww_2\in \bR^m$. Now $D=\lambda s^2\xx_2\xx_2\trans- t^2\ww_2\ww_2\trans=0$, so $t=\sqrt\lambda s\ne 0$, $s\le \lambda^{-1/2}$ and $\ww_2=c \xx_2$ for some $c\in \{-1,1\}$, so we have
 $$\ww=\npmatrix{\sqrt{1-\lambda s^2}\ww_1\\ c \sqrt\lambda s\xx_2}.$$
 To see that $\vv$ is not generically realisable for $K_2\vee mK_1$, consider the set $\mc Y:=\{\yy_{-1},\yy_{1}\}$, where $\yy_{k}:= \e_1+k\sqrt \lambda\e_{m+2}$ for $k\in \{-1,1\}$. Any orthogonal matrix $U$ with $U\trans X U=D_{\{-1,0,\lambda\},\vv}$ has first and last columns equal to $a \ww$ and $b \xx$, respectively, for some $a,b\in \{-1,1\}$. Removing the first two entries of the vector $U\yy_{k}=a\ww + bk\sqrt{\lambda}\xx$ leaves $(ac  + bk )\sqrt{\lambda}s \xx_2=\mathbf{0}_m$ for $k=-abc$, so $U\yy_{-abc}$ is not nowhere-zero. This proves that $\vv$ is not generically realisable for $K_2\vee mK_1$.

 To prove that $\vv$ is sane for $K_2\vee mK_1$, it suffices for each $\lambda>0$ to  construct a matrix in $S(K_2\vee mK_1)$ with a nowhere-zero orthonormal eigenbasis, and eigenvalues $\lambda$ and  $-1$ with multiplicity $1$, and $0$ with multiplicity $m$. To deal with the $\lambda \neq 1$ case, consider the unit vectors $\xx_1=\frac1{\sqrt2}\1_2$ and $\xx_2=\frac1{\sqrt m}\1_m$, and the orthonormal vectors
\[\xx:=\frac1{\sqrt{1+\lambda}}\npmatrix{\sqrt{\lambda} \xx_1\\\xx_2}, \quad
   \ww:=\frac1{\sqrt{1+\lambda}}\npmatrix{-\xx_1\\\sqrt{\lambda}\xx_2}.\]
The rank two matrix $X:=\lambda \xx\xx\trans-\ww\ww\trans$ is then given by
\[ X=\npmatrix{(\lambda-1)\xx_1\xx_1\trans & \sqrt{\lambda}\xx_1\xx_2\trans\\
    \sqrt{\lambda}\xx_2\xx_1\trans& 0_m}\in S(K_2\vee mK_1).\]
%
%
It is easy to find a nowhere-zero vector orthogonal to both $\xx$ and $\ww$, so by Lemma~\ref{lem:bases}, the kernel of $X$ has a nowhere-zero orthonormal basis. Appending $\ww$ and $\xx$ gives a nowhere-zero orthonormal eigenbasis for $X$, as required.

The final case, when $\lambda=1$, can be resolved by the matrix $X:=\npmatrix{A&B\\B\trans&0_m}$ where $A=\tfrac1{20}\npmatrix{4&3\\3&-4}$ and $B=\tfrac{\sqrt{3}}{4\sqrt{m}}\npmatrix{\1_m\trans\\2\cdot \1_m\trans}$.
We can check that $X$ has rank $2$, nonzero eigenvalues $1$ and $-1$, and a nowhere-zero orthonormal eigenbasis, by straightforward calculation.
\end{example}

\section{Joins of unions of complete graphs}\label{sec:complete}

Proposition \ref{complete-generic} together with the solution to the IEP-G for complete graphs, emphasises that the spectra of matrices in $S(K_n)$ are the least constrained. We exploit this fact to determine $q$ of joins of unions of complete graphs, and derive some conditions for $q \neq 2$ in a more general setting.

\begin{proposition}\label{complete-generic}
  For $n\in \bN$, the complete graph $K_n$ is generically realisable.
\end{proposition}
\begin{proof}
  This is trivial for $n=1$. Let $n\ge2$ and recall (Remark \ref{0-1wb}) that any multiplicity list with sum~$n$ and at least two nonzero entries is spectrally arbitrary for $K_n$. Hence, if $D$ is any $n\times n$ diagonal matrix with $D\ne aI_n$, $a\in \bR$, then there is an orthogonal matrix $U$ so that $A:=UDU\trans \in S(K_n)$. For any finite set $\mc Y\subseteq \bR^n\setminus\{0\}$ we can find an orthogonal $V$ arbitrarily close to~$I_n$ so that $VU\yy$ is nowhere-zero for every $\yy\in \mc Y$. Since $A\in S(K_n)$, 
  for $V$ sufficiently close to~$I_n$, we also have $VAV\trans \in S(K_n)$. Hence, the multiplicity list of $D$ is generically realisable for $K_n$.
\end{proof}

We now introduce some notation. Assume throughout this section that $k, \ell \geq 1$ are positive integers. Let $\m=(m_1,\ldots,m_k)\in\bN^k$ denote a $k$-tuple of natural numbers.
We abbreviate a tuple such as $(m_1,\dots,m_s,a,\dots,a)$, with $a$ appearing $t$ times, as $(m_1,\dots,m_s,a^t)$.
We define $\km:=\bigcup_{i\in [k]}K_{m_i}$,
and $|\m|:=\sum_{i\in [k]}m_i=|\km|$. Clearly, $k\le |\m|$. Recall that $\iso (\km)$ denotes the set of isolated vertices of the graph $\km$, and let $\isa{\m}=|\{i\in [k] \colon m_i=1\}|=|\iso(\km)|$.
  Note that
  $ |\m|+\isa{\m}\geq 2k$.
    In particular, if $\isa{\m}=0$, then $|\m|\geq 2k$.

Next, we wish to prove that the number of distinct eigenvalues of a join of two unions of complete graphs can only be equal to $2$ or to $3$. To see this, we will use the following simple lemma, which follows from \cite[Theorem~3]{2017arxiv170802438} or \cite[Lemma~2.9]{MR3891770}.

\begin{lemma}\label{lem:jdup}
  Let $G$ be a graph with at least one edge, and let $v\in V(G)$. Let $H$ be the graph obtained from $G$ by replacing $v$ with a complete graph~$K_n$ and adding edges from every neighbour of $v$ in $G$ to every vertex of~$K_n$. Then $q(H)\le q(G)$.
\end{lemma}

 \begin{corollary}\label{cor:complete-jdup}
   Let $\m, \tuple m' \in \bN^k$, $\n, \tuple n' \in \bN^\ell$, where 
   $\tuple n'\ge \tuple n$ and $\tuple m'\ge \tuple m$.
If $q (\km \vee \kn)=2$, then $q (\kx{\m'} \vee \kx{\n'})=2$.
\end{corollary}
\begin{proof}
  The graph $\kx{\m'} \vee \kx{\n'}$ may be obtained from $\km\vee\kn$ by choosing a set of $k+\ell$ vertices (one from each component of $\km$ and one from each component of $\kn$) and applying the procedure described in Lemma~\ref{lem:jdup} successively to each one.
\end{proof}

\begin{corollary}\label{cor:complete}
  If $\m\in\bN^k$ and $\n\in\bN^\ell$, then
  \[ q(\km\vee \kn)=
    \begin{cases}
      2&\text{if $\km$ and $\kn$ have compatible multiplicity matrices,}\\
      3&\text{otherwise}.
    \end{cases}\]
\end{corollary}
\begin{proof}
  Since $\km\vee\kn$ has at least one edge, $q(\km\vee\kn)\ge 2$. Moreover, $q(k K_1  \vee \ell K_1)=q(K_{k,l})\in \{2,3\}$,         by \cite[Corollary~6.5]{MR3118943}. Applying Lemma~\ref{lem:jdup} a total of $k+\ell$ times, we obtain
  $2\le q(\km \vee \kn)\leq q(k K_1  \vee \ell K_1)\le 3$,
  so $q(\km\vee\kn)\in \{2,3\}$.

  By Proposition \ref{complete-generic}, $\km$ and $\kn$ are generically realisable (and, hence, they are also sane). By Corollary~\ref{compat-iff}, $q(\km\vee\kn)=2$ if and only if $\km$ and $\kn$ have compatible multiplicity matrices.
  %
\end{proof}

\begin{remark}
 Corollary \ref{cor:complete-jdup} can also be deduced easily from Corollary~\ref{cor:complete} by considering multiplicity matrices.
Given $\m,\m',\n,\n'$ as in Corollary \ref{cor:complete-jdup}, suppose
 $V \in \bN_0^{r \times k}$ and $W \in \bN_0^{r \times \ell}$ are compatible multiplicity matrices for $\km$ and $\kn$,  respectively. Then $V':=V+\e_1(\m'-\m)\trans$ and $W':=W+\e_1(\n'-\n)\trans$ are  multiplicity matrices for $\kx{\m'}$ and $\kx{\n'}$, respectively, and  $\widetilde V=\widetilde{V'}$ and $\widetilde W=\widetilde{W'}$, so $V'$ and $W'$ are compatible. Hence, if $q(\km\vee\kn)=2$, then $q(\kx{\m'}\vee\kx{\n'})=2$.
%
%
 \end{remark}


%
%
%

The following easy observations about multiplicity matrices are worth recording at this point.

\begin{proposition}\label{prop:mult-complete}
  Let $k,r\in \bN$ with $r\ge3$.
  \begin{enumerate}
  \item   For $\m\in \bN^k$, a matrix $V\in \bN_0^{r\times k}$ is a multiplicity matrix for $\km$ if and only if, for every $i\in [k]$, we have $\1_r\trans V\e_i=m_i$, and if $m_i\ge 2$, then $V\e_i$ has at least two nonzero entries.
  \item If $V$ is a multiplicity matrix for a graph $G=\bigcup_{i=1}^kG_i$, where each $G_i$ is connected, 
    then $V$ is a multiplicity matrix for $\km$ where $\m:=(|G_1|,\dots,|G_k|)$.
  \item If $V$ and $W$ are compatible matrices, then $\widetilde V$ and $\widetilde W$ have at least one nonzero entry in each column.
  \end{enumerate}
\end{proposition}

\begin{corollary}\label{cor:K-to-connected-q=3}
  Let $G=\bigcup_{i=1}^k G_i$, $H=\bigcup_{j=1}^\ell H_j$, $\m=(m_1,\ldots,m_k)$ and $\n=(n_1,\ldots,n_\ell)$, where $G_i$ and $H_j$ are connected graphs with  $|G_i|=m_i$ and $|H_j|=n_j$. 
  If $q(\km \vee   \kn)= 3$, then $q(G\vee H)\geq 3$.
\end{corollary}
\begin{proof}
  Since $G\vee H$ has at least one edge, $q(G\vee H)\ge2$. If $q(G\vee H)=2$, then by Theorem~\ref{thm:q=2}, there is a pair $V,W$ of compatible multiplicity matrices for $G$ and $H$. By Proposition~\ref{prop:mult-complete}, these are multiplicity matrices for $\km$ and $\kn$, respectively. Hence, by 
Corollary~\ref{cor:complete},
we have $q(\km\vee\kn)= 2$, a contradiction.
\end{proof}

\begin{remark}
 In contrast with complete graphs, paths $P_n$ are realisable by matrices $A \in S(P_n)$ with the most constrained spectrum.  We will explore this in our follow up paper \cite{LOS21}, where it will be shown that $q(\bigcup_{i\in [k]}P_{m_i} \vee  \bigcup_{j\in [\ell]}P_{n_j})=2$ implies $q(G\vee H)=2$,
where $G$ and $H$ are as in Corollary~\ref{cor:K-to-connected-q=3}.
\end{remark}

 The following straightforward lemma is used several times in our arguments below.  The proof is left to the reader.

  \begin{lemma}\label{lem:major}
    If $s,t \in \bN$ and $\tuple{c} \in \bN^s$, then there exists $\tuple{b} \in \bN^s$ with $|\tuple{b}|=t$ and $\tuple{b} \leq \tuple{c}$ if and only if $s \leq t\leq |\tuple{c}|$.
\end{lemma}

 Now that we have established the underlying framework, we are left with the combinatorial question of deciding under what conditions on $\m$ and $\n$ compatible multiplicity matrices for $\km$ and $\kn$ exist. Our first result in this direction gives a sufficient condition. 


\begin{proposition}\label{prop:q2-compatible matrices}
   For $\m\in\bN^k$, $\n\in \bN^{\ell}$,  the following conditions are equivalent:
  \begin{enumerate}
  \item \label{a1} $k+\ell\le \min\{|\m|+\isa{\m},|\n|+\isa{\n}\}$.
  \item \label{a2} $\km$ and $\kn$ have a compatible pair of multiplicity matrices in $\bN_0^{3\times k}$ and $\bN_0^{3\times \ell}$, respectively;
  \end{enumerate}
  Moreover, if either \ref{a1} or \ref{a2} is valid, then $q(\km\vee \kn)=2$.
\end{proposition}

\begin{proof}
  Suppose \ref{a2} holds and $V\in \bN_0^{3\times k}$, $W\in \bN_0^{3\times \ell}$ are compatible multiplicity matrices for $\km$, $\kn$, respectively. Let $p:=\widetilde V\1_k=\widetilde W\1_\ell\in \bN$. Since $\widetilde V\trans\widetilde W>0$, every entry of the row vectors $\widetilde V$ and $\widetilde W$ is at least $1$, so $p\ge\max\{ k,\ell\}$. Every column of $V$ corresponds to a connected component $K_{m_i}$ of $\km$ and is of the form $(a_i,b_i,c_i)\trans$ where $b_i\ge1$ and $a_i+b_i+c_i=m_i$. If $m_i=1$ for some $i$, 
  then the corresponding column is necesarily equal to $(0,1,0)\trans$. For columns with $m_i>1$ we have $a_i+c_i\ge1$ by Proposition~\ref{prop:mult-complete}, so $1 \leq b_i\le m_i-1$. Summing up, we obtain
  $  p=\sum_{i\in [k]}b_i\le \isa{\m}+\sum_{i:m_i>1}(m_i-1)=|\m|+\isa{\m}-k$,
and similarly $p\le |\n|+\isa{\n}-\ell$. Since $\max\{k,\ell\} \le p$, we get~\ref{a1}.

  Now suppose the inequality \ref{a1} holds, and without loss of generality assume $k \leq \ell$.
  Then $k \leq \ell\leq |\m|-k +\isa{\m}=\sum_{i\in [k]}\max\{1,m_i-1\}$, and there exist $t_1,\dots,t_k\in \bN$ with $t_i\le \max\{1,m_i-1\}$ for $i\in [k]$ and $\sum_{i\in [k]} t_i=\ell$, by Lemma \ref{lem:major}.
Now
\[ V:=\npmatrix{m_1-t_1&\dots&m_k-t_k\\t_1&\dots&t_k\\0&\dots&0}\in \bN_0^{3\times k}\text{ and } W:=\npmatrix{n_1-1&\dots&n_\ell-1\\1&\dots&1\\0&\dots&0}\in \bN_0^{3\times \ell}\]
 are compatible multiplicity matrices for $\km$ and $\kn$, respectively. Hence \ref{a2} holds.

The final claim is immediate from Corollary~\ref{cor:complete}.
\end{proof}

The next result follows from \cite[Theorem~4.4]{MR3118943}. We present an alternative proof using our methods.

\begin{proposition}\label{prop:comp vs vertices}
Let $\m\in\bN^k$, $\n\in \bN^{\ell}$. If $|\m| < \ell$, then
$q(\km\vee\kn)= 3$.
 \end{proposition}


 \begin{proof}
Otherwise, by Corollary \ref{cor:complete}, there exist compatible multiplicity matrices $V\in \bN_0^{r \times k}$ and $W\in \bN_0^{r \times \ell}$ for $\km$ and $\kn$, respectively. Then $\widetilde V\trans \widetilde W>0$, which implies that $\widetilde W$ has no zero columns, so $\1_{r-2}\trans \widetilde W \geq \1_{\ell}\trans$. Since $\widetilde V \1_k=\widetilde W \1_\ell$, it follows that
$\ell =\1_{\ell}\trans \1_\ell \leq  \1_{r-2}\trans \widetilde W \1_\ell=\1_{r-2}\trans \widetilde V \1_k \leq \1_r\trans V \1_k= |\km|=|\m|$,
a contradiction.
 \end{proof}

Note that there is a gap between the sufficient conditions for $q=2$ given in Proposition~\ref{prop:q2-compatible matrices} and the necessary conditions for $q=2$ that follow from Proposition \ref{prop:comp vs vertices}. It turns out that isolated vertices play an important role in the complete solution, and we consider different cases that can occur below.

 \subsection{No isolated vertices}\label{subsec:no iso}
First we examine the case when at least one of $\km$ and $\kn$ has no isolated vertices. In particular, when neither one of these graphs has an isolated vertex we will see in Proposition~\ref{prop:complete no iso} that the sufficient condition in Proposition~\ref{prop:comp vs vertices} for $q(\km\vee \kn)$ to be 3  
is also necessary. 

\begin{lemma}
  \label{lemma:H=2k}
Let $k\le \ell$, $\m\in\bN^k$ and $\n\in \bN^{\ell}$ with $\isa{\m}=0$. If $|\n|=2k$, then $q(\km\vee\kn)=2$.
\end{lemma}
\begin{proof}
  By permuting the entries of $\n$ if necessary, we may assume that $\n=(n_1,\ldots,n_t,1^{\isa{\n}})$, where $t:=\ell-\isa{\n}$ and $n_i \geq 2$ for $i\in [t]$. Write $\isa{\n}=2a+b$ for $b \in \{0,1\}$. From $2k=|\n|=\iota(\n)+\sum_{i\in [t]} n_i$, we get
$k=a+\frac12(b+\sum_{i\in [t]}n_i)\ge a+\frac12b+t$, and hence $k\ge a+b+t$.
  Denote $\tuple{R}:=(r_1,r_2)=(k-t-a-b,k-t-a)$, $\tuple{C}:=(c_1,\ldots,c_t)=(n_1-2,\ldots,n_t-2)$, and note that $r_1+r_2=\sum_{j\in[t]}c_j$. Hence, by \cite[Theorem~2.1.2]{MR2266203}, there exists a matrix $Y=[y_{ij}]\in \bN_0^{2\times t}$ with row sums equal to $\tuple{R}$ and column sums equal to $\tuple C$.  Since $\isa{\m}=0$, we have $m_i\ge2$ for $i\in [k]$. Define compatible matrices $V \in \bN_0^{4\times k}$ and $W\in \bN_0^{4\times \ell}$ as follows: \[ V:=
  \npmatrix{
    m_1-2&\dots& m_k-2\\
    1 &\dots & 1\\
    1& \dots & 1\\
    0& \dots & 0
  },\quad
  W:=
  \npmatrix{
    0    & \dots & 0    & 0             & 0           \\
    y_{11}+1 & \dots & y_{1t}+1 & \1_{a+b}\trans & 0           \\
    y_{21}+1 & \dots & y_{2t}+1 & 0             & \1_{a}\trans \\
    0        & \dots & 0        & 0             & 0
  }.\]
%
By Corollary~\ref{prop:mult-complete}, $V$ and $W$ are multiplicity matrices for $\km$ and $\kn$, respectively, so $q(\km\vee\kn)=2$ by Corollary~\ref{cor:complete}.
%
\end{proof}

\begin{proposition}
  \label{prop:no iso G}
Let $k\le \ell$, $\m\in\bN^k$ and $\n\in \bN^{\ell}$ with $\isa{\m}=0$.
If $2k\leq \ell \leq |\m|$ or  $\ell \leq 2k\leq |\n|$, then $q(\km \vee \kn)=2$.
\end{proposition}

\begin{proof}
  Suppose first that $2k\leq \ell \leq |\m|$. Then $k\le \ell-k\le |\m|-k$, and since $\isa{\m}=0$, we have $m_i\ge 2$ for every $i\in [k]$. By Lemma~\ref{lem:major}, there exist $r_i\in [m_i-1]$ so that  $\ell-k =\sum_{i\in [k]} r_i$. 
  The matrices 
\[{\small\npmatrix{m_1-r_1-1&\dots&m_k-r_k-1\\1&\dots&1\\r_1&\dots&r_k\\0&\dots&0}}\text{\,and\,}
{\small\npmatrix{%
      n_1-1&\dots&n_k-1&n_{k+1}-1&\dots&n_\ell-1\\
      1&\dots&1&0&\dots&0\\
      0&\dots&0&1&\dots&1\\
      0&\dots&0&0&\dots&0}}\]
are compatible multiplicity matrices for $\km$ and $\kn$, respectively, hence $q(\km \vee \kn)=2$ by Corollary \ref{cor:complete}.

Suppose instead that $\ell \leq 2k\leq |\n|$. By Lemma~\ref{lem:major}, there is a vector $\p\in \bN^\ell$ with $\p\le \n$ and $|\p|=2k$.
By Lemma \ref{lemma:H=2k} we have $q(\km\vee\kx{\p})=2$
and it follows by Corollary~\ref{cor:complete-jdup} that $q(\km\vee\kn)=2$.
\end{proof}

Assuming $\isa{\m}=\isa{\n}=0$,  we get the first complete resolution of $q(\km \vee \kn)$.

\begin{proposition}
  \label{prop:complete no iso}
 Let $\m\in\bN^k$, $\n\in \bN^{\ell}$. If  $\isa{\m}=\isa{\n}=0$, then
 $$q(\km \vee \kn)=\begin{cases}
   2&\text{if $\max\{k,\ell \} \leq \min\{|\m|,|\n|\}$},\\
   3 & \text{otherwise.}
   \end{cases}$$
\end{proposition}

\begin{proof}
  By symmetry we may assume without loss of generality that $k \leq \ell$. Under this assumption, we must prove that $q(\km\vee\kn)=2$ if and only if $\ell \le |\m|$.
  If $|\m| < \ell$, then $q(\km \vee \kn)=3$ by Proposition \ref{prop:comp vs vertices}, and
  if $2k \leq \ell \leq |\m|$, then $q(\km \vee \kn)=2$ by Proposition \ref{prop:no iso G}.

  The only remaining case is when $\ell \leq |\m|$ and
  $k \leq \ell < 2k$. 
  Then $1\le 2k-\ell\leq k\leq \ell$ and  $\isa{\m}=\isa{\n}=0$ implies that $m_i,n_j\ge2$ for all $i\in [k]$ and $j\in [\ell]$. We define matrices $V\in \bN_0^{4\times k}$ and $W\in \bN_0^{4\times \ell}$ as follows:
  \[ V:=\npmatrix{m_1-2&\dots &m_k-2\\1&\dots&1\\1&\dots&1\\0&\dots&0}\]
  and
  \[ W:=\small\npmatrix{
      n_1-2&\dots&n_{2k-\ell}-2&n_{2k-\ell+1}-1&\dots&n_{k}-1&n_{k+1}-1&\dots&n_{\ell}-1\\
      1&\dots& 1&1 &\dots &1 &0 &\dots&0 \\
      1&\dots& 1&0 &\dots &0 &1 &\dots& 1\\
      0&\dots&0 &0 &\dots &0 &0 &\dots& 0
      }.
    \]
Since $V,W$ are compatible multiplicity matrices for $\km$ and $\kn$, respectively, we conclude that $q(\km\vee\kn)=2$ by Corollary~\ref{cor:complete}.
\end{proof}

\subsection{Unions of complete graphs with isolated vertices}\label{subsec:iso}


When both $\km$ and $\kn$ have isolated vertices, we have the following characterisation.

\begin{proposition}\label{prop:q2-iso}
   Let $\m\in\bN^k$, $\n\in \bN^{\ell}$. If $\isa{\m}>0$ and $\isa{\n}>0$, then the following conditions are equivalent:
  \begin{enumerate}
  \item \label{b1} $q(\km\vee \kn)=2$;
  \item \label{b2} $\km$ and $\kn$ have a compatible pair of multiplicity matrices;
  \item \label{b3} $\km$ and $\kn$ have a compatible pair of multiplicity matrices in $\bN_0^{3\times k}$ and $\bN_0^{3\times \ell}$, respectively;
  \item \label{b4} $k+\ell\le \min\{|\m|+\isa{\m},|\n|+\isa{\n}\}$.
  \end{enumerate}
\end{proposition}
\begin{proof}
  \ref{b1} and \ref{b2} are equivalent by Corollary~\ref{cor:complete}. \ref{b3} and \ref{b4} are equivalent by Proposition \ref{prop:q2-compatible matrices} and plainly imply \ref{b1} and \ref{b2}.

  Suppose now \ref{b2} holds, and $V \in \R^{r \times k}$ and $W \in \R^{r \times \ell}$ are compatible multiplicity matrices for $\km$ and $\kn$, respectively, for some $r\ge3$.   Since $\isa{\m}>0$ and $\isa{\n}>0$, the condition $\widetilde V\trans\widetilde W>0$ implies that for some $i\in [r-2]$, $i$th rows of both  $\widetilde V$ and $\widetilde W$ are nowhere-zero. Define $V'\in \bN_0^{3\times k}$ and $W'\in \bN_0^{3\times \ell}$ as follows:
\[ V':=\npmatrix{\1_r\trans V - \e_i\trans\widetilde V\\\e_i\trans\widetilde V\vphantom{\displaystyle\int}\\\mathbf{0}_k\trans},\quad
  W':=\npmatrix{\1_r\trans W - \e_i\trans\widetilde W\\\e_i\trans\widetilde W\vphantom{\displaystyle\int}\\\mathbf{0}_{\ell}\trans}.\]
Note that $V'$ has the same column sums as $V$, and if the $j$th column of $V$ has more than one non-zero entry, then the $j$th column of $V'$ has (precisely) two non-zero entries, so by Proposition~\ref{prop:mult-complete}, $\widetilde{V'}$ is a multiplicity matrix for $\km$; similarly, $W'$ is a multiplicity matrix for $\kn$. Since $\widetilde {V'}=\e_i\trans\widetilde V$ and $\widetilde{W'}=\e_i\trans\widetilde W$, it follows that $V'$ and $W'$ are compatible multiplicity matrices for $\km$ and $\kn$, so \ref{b2} implies \ref{b3}.
\end{proof}

Suppose now  $k \leq \ell$ and $\isa{\n}=0$.  Under this assumption we have $ |\n|\geq 2 \ell \geq k+\ell$. Hence the sufficient condition \ref{a1} in Proposition \ref{prop:q2-compatible matrices} for $q(\km \vee \kn)=2$  is equivalent to $k+\ell\leq |\m|+\isa{\m}$. We will show that this condition is also necessary for $q(\km \vee \kn)=2$ in the case $\isa{\m}>0$.

\begin{lemma}\label{lem:rowbound}
  If $r\ge3$ and $V\in\bN_0^{r\times k}$ is a multiplicity matrix for a graph~$G$ with $k$ connected components and $\m=\1_r\trans V$ is the vector of connected component sizes of $G$, then for any $s\in [r-2]$, we have $\e_s\trans \widetilde V\1_k\le \isa{\m}+|\m|-k$.
\end{lemma}
\begin{proof}
  We have
  \[  \e_s\trans\widetilde V\1_k
    =\sum_{i:m_i=1}\e_s\trans \widetilde V\e_i+\sum_{i:m_i>1}\e_s\trans \widetilde V\e_i
    \leq \isa{\m}+\sum_{i:m_i>1}(m_i-1)=\isa{\m}+|\m|-k.\qedhere\]
\end{proof}

\begin{proposition}
  \label{prop:G yes H no}
 Let $k\le\ell$, $\m\in\bN^k$ and $\n\in \bN^{\ell}$. If $\isa{\m}>0$ and $\isa{\n}=0$, then
 $$q(\km \vee \kn)=\begin{cases}
   2 & \text{if $k+\ell\leq |\m| +\isa{\m}$},\\
   3 & \text{otherwise.}
   \end{cases}
   $$
\end{proposition}
\begin{proof}
 Since $k \leq \ell$ and $\isa{\n}=0$, we have $|\n|\geq 2\ell \geq k+\ell$ and thus the sufficiency of the condition  $k+\ell\leq |\m| +\isa{\m}$ follows by Proposition \ref{prop:q2-compatible matrices}.

 Conversely, if $q(\km\vee\kn)=2$, then by Corollary \ref{cor:complete} there exist compatible multiplicity matrices $V\in \bN_0^{r \times k}$ and $W\in \bN_0^{r \times \ell}$ for $\km$ and $\kn$, respectively. Since $\isa{\m}>0$, some column of $\widetilde V$ is equal to $\e_s$, where $s\in [r-2]$. Since $\widetilde V\trans \widetilde W>0$, this implies that $\e_s\trans\widetilde W>0$, so $\e_s\trans \widetilde W \ge \1_{\ell}\trans$. By Lemma~\ref{lem:rowbound},
  \[ \ell = \1_\ell\trans \1_\ell\le \e_s\trans\widetilde W\1_\ell=\e_s\trans\widetilde V\1_k\le \isa{\m}+|\m|-k,\]
  so $k+\ell\le \isa{\m}+|\m|$.
\end{proof}

\begin{proposition}
  \label{prop:G no H yes}
 Let $k\le\ell$, $\m\in\bN^k$ and $\n\in \bN^{\ell}$. If $\isa{\m}=0$ and $\isa{\n}>0$,
then
$$q(\km \vee \kn)=\begin{cases}
2& \text{if $2k \leq \ell \leq |\m|$, or $\ell \leq 2k \leq |\n|$, or $k+\ell\leq |\m|$},\\
3 & \text{otherwise.}
\end{cases}$$
\end{proposition}

\begin{proof}
 If $2k\leq \ell \leq |\m|$ or $\ell \leq 2k \leq |\n|$ or $k+\ell\leq |\m|$, then $q(\km \vee \kn)=2$ by Propositions~\ref{prop:q2-compatible matrices} and~\ref{prop:no iso G}.
If $\ell>|\m|$, then $q(\km \vee \kn)=3$ by Proposition \ref{prop:comp vs vertices}.

 If none of the above relations holds, then in particular, we have $|\n|<2k$ and $|\m| <k+\ell$.
  If $q(\km \vee \kn)=2$, then by Corollary \ref{cor:complete} there exist $r\ge3$ and compatible multiplicity matrices $V\in \bN_0^{r \times k}$ and $W\in \bN_0^{r \times \ell}$ for $\km$ and $\kn$, respectively. Since $\isa{\n}>0$, we can assume without loss of generality that $n_\ell=1$ and that $\widetilde W\e_\ell=\e_1\in \bN_0^{r-2}$. The condition $\widetilde V\trans \widetilde W >0$ now implies that $\e_1\trans\widetilde V \geq \1_k\trans$.
  Lemma~\ref{lem:rowbound} together with $\isa{\m}=0$ gives
   $\e_1\trans \widetilde W \1_\ell=\e_1\trans \widetilde V \1_k \leq |\m|-k<\ell$.
  If $\e_1\trans\widetilde W\ge \1_\ell\trans$, then we obtain an immediate contradiction. So the first row of $\widetilde W$ contains at least one zero entry. Hence, each of the $k$ columns of $\widetilde V$ necessarily has at least one nonzero entry in a row other than row~$1$, in addition to the nonzero entry in row~$1$, to satisfy $\widetilde V\trans \widetilde W>0$.
    %
So $2k\le \1_{r-2}\trans \widetilde V \1_k=\1_{r-2}\trans \widetilde W\1_\ell\le \1_rW\1_\ell=|\n|$, a contradiction. This proves $q(\km\vee\kn)=3$.
\end{proof}

We remark that by the constructions in Subsections~\ref{subsec:no iso} and~\ref{subsec:iso}, whenever $q(\km\vee\kn)=2$, there exist compatible multiplicity matrices for $\km$ and $\kn$ with at most four rows.

\subsection{Complete result}

We summarise the different cases that we have considered in a theorem that completely resolves $q(\km\vee \kn)$, and follows immediately from Propositions~\ref{prop:complete no iso}, \ref{prop:q2-iso}, \ref{prop:G yes H no} and \ref{prop:G no H yes}.

\begin{theorem}\label{thm:general}
Let $k\le\ell$, $\m\in\bN^k$ and $\n\in \bN^{\ell}$. We have $q(\km\vee \kn)=2$ if and only if one of the following is true:
\begin{enumerate}
\item\label{g1} $\isa{\m}=0$, $\isa{\n}= 0$ and  $\ell \le |\m|$,
 \item\label{g2} $\isa{\m}>0$ and $k+\ell\le |\m|+\isa{\m}$, or
\item\label{g3} $\isa{\m}=0$, $\isa{\n}>0$ and either $k+\ell\leq |\m|$, or $2k \leq \ell \leq |\m|$, or $\ell \leq 2k \leq |\n|$.
\end{enumerate}
Otherwise, $q(\km\vee \kn)=3$.
\end{theorem}

Theorem~\ref{thm:main} is a straightforward consequence of Theorem \ref{thm:general}
and Corollary~\ref{cor:K-to-connected-q=3}.
Note that Theorem \ref{thm:general} allows us to observe properties of $q(\km \vee \kn)$. For example,
for any two fixed numbers of connected components $k$ and $\ell$, we can always choose $\m \in \bN^k$ and $\n \in \bN^\ell$, with $|\m|$ and $|\n|$ sufficiently large to achieve $q(\km \vee \kn)=2$.

\subsection{Multiplicities}

Let $X$ be an $n \times n$ orthogonal symmetric matrix, and let $i_+(X)$ and $i_-(X)$ denote the multiplicities of $1$ and $-1$ as eigenvalues of $X$, respectively. 
The study of $i_+(X)$ can be motivated by the fact that the map $X\mapsto P(X)=\frac12(I+X)$ is a bijection between the $n\times n$ orthogonal symmetric matrices and the orthogonal projections onto subspaces of $\bR^n$, and $\rank P(X)=i_+(X)$. Moreover, for any graph $G$, we have $X\in S(G)$ if and only if $P(X)\in S(G)$.

Let $G$ and $H$ be graphs such that $q(G \vee H)=2$. Let $X\in S(G\vee H)$ be an orthogonal matrix with the corresponding compatible multiplicity matrices $V,W$ guaranteed by Theorem~\ref{thm:q=2}. Let  $\mu=\1_{r-2}\trans \widetilde{V}\1_k=\1_{r-2}\trans\widetilde{W}\1_\ell$. 
By the proof of Theorem~\ref{thm:q=2}, we have 
$i_+(X)\geq \mu$ and $n-i_+(X)=i_-(X)\geq \mu$, implying $\mu\leq i_+(X)\leq n-\mu$.


To consider the special case $G=\km$ and $H=\kn$, let $\mu(\m,\n)$ denote the minimum value of $\1_{r-2}\trans \widetilde{V}\1_k=\1_{r-2}\trans \widetilde{W}\1_\ell$ over all compatible multiplicity matrices $V$ and $W$ for $\km$ and $\kn$. By the previous paragraph, we have $\mu(\m,\n)\le i_+(X)\le |\m|+|\n|-\mu(\m,\n)$ for any orthogonal $X\in S(\km\vee \kn)$. Moreover, if $V_0 \in \bN_0^{r\times k}$ and $W_0 \in \bN_0^{r\times \ell}$ are compatible multiplicity matrices for $\km$ and $\kn$, then so are any matrices $V \in \bN_0^{r\times k}$ and $W \in \bN_0^{r\times \ell}$ satisfying $\widetilde{V}=\widetilde{V_0}$, $\widetilde{W}=\widetilde{W_0}$,
$(\e_1+\e_{r})\trans V=(\e_1+\e_{r})\trans V_0$ and $(\e_{1}+\e_{r})\trans W=(\e_{1}+\e_{r})\trans W_0$. By the construction of Theorem~\ref{thm:compatible sane}, this implies that for \emph{every} integer $i$ with $\mu(\m,\n)\le i\le  |\m|+|\n|-\mu(\m,\n)$, we have $i=i_+(X)$ for some orthogonal matrix $X\in S(\km\vee\kn)$.

Let  $k \leq \ell$, $\m\in\bN^k$ and $\n\in \bN^{\ell}$, then:
\begin{equation}\label{eq:mu bound}
\ell \leq \mu(\m,\n) \leq \max\{\ell,2k\}.
\end{equation}
The first inequality is clear, and the second one follows by examining the constructions of multiplicity matrices in the proof of Proposition \ref{prop:q2-compatible matrices}, and in the proofs in Subsections~\ref{subsec:no iso} and~\ref{subsec:iso}.

\begin{lemma}\label{lem:k+l}
  Let $k,\ell \in \bN$, $k \leq \ell$, $\m\in \bN^k$, $\n \in \bN^\ell$. If  $k+\ell >|\m|+\isa{\m}$, then 
  $\mu(\m,\n) \geq 2k$.
\end{lemma}

\begin{proof}
  %
  Suppose $\mu(\m,\n)<2k$ and there exist $V,W$ as above with 
  $\1_{r-2}\trans \widetilde V\1_k <2k$. Some column of $\widetilde V$ then has sum less than~$2$, so that column has exactly one non-zero entry, say in row $s$. Since $\widetilde V\trans\widetilde W>0$, this implies that row $s$ of $\widetilde W$ is nowhere-zero. By Lemma \ref{lem:rowbound} and the compatibility of $V$ and $W$, we have $|\m|+\isa{\m}-k \geq \e_s\trans \widetilde V \1_k=\e_s\trans \widetilde W \1_\ell\geq \ell$, a contradiction. So $\mu(\m,\n)\ge 2k$.
\end{proof}

In the next theorem we show that $\mu(\m,\n)$ always reaches one of the two upper bounds given in \eqref{eq:mu bound}.

\begin{theorem}\label{thm:mu}
Let $k\le\ell$, $\m\in\bN^k$ and $\n\in \bN^{\ell}$. If $q(\km \vee \kn)=2$, then
  $$\mu(\m,\n)=\begin{cases}
    2k & \text{if }  |\m|+\isa{\m}-k<\ell<2k,\\
    \ell & \text{otherwise.}
 \end{cases}$$
Consequently, there exists a matrix $A\in S(\km \vee \kn)$ with $q(A)=2$ whose eigenvalues have multiplicities $\{t,|\m|+|\n|-t\}$ if and only if $\mu(\m,\n) \leq t \leq |\m|+|\n|-\mu(\m,\n)$.
\end{theorem}

\begin{proof}
  If $|\m|+\isa{\m}-k<\ell<2k$, then we have $2k\le \mu(\m,\n)\le\max\{\ell,2k\}=2k$ by Lemma~\ref{lem:k+l} and \eqref{eq:mu bound}, so $\mu(\m,\n)=2k$.
  On the other hand, if $\ell\ge 2k$, then $\mu(\m,\n)=\ell$ by~\eqref{eq:mu bound}, and if $|\m|+\isa{\m}-k\ge \ell$, then the matrix $W$ constructed in the the proof of Proposition~\ref{prop:q2-compatible matrices} shows that $\mu(\m,\n)\le \ell$, so $\mu(\m,\n)=\ell$ by \eqref{eq:mu bound}.
  The final claim regarding multiplicities follows immediately from the discussion before Lemma~\ref{lem:k+l}.
 \end{proof}

Note that both cases in Theorem~\ref{thm:mu} occur, for example when $k=2$ and $\ell=3$ we have $\mu((K_3 \cup K_2) \vee(K_2 \cup 2K_1))=\ell$ and $\mu(2K_2 \vee(K_2 \cup 2K_1))=2k$.

\subsection{Examples}

We conclude this work with two examples, that illustrate the strength of the general result.

%

\begin{example}\label{ex:Km connected}
\label{cor:joinK_m}
For $m \in \bN$ and $\n \in \bN^{\ell}$, we have
$$q(K_m \vee \kn)=\begin{cases}
 2& \text{if $\ell \leq m$},\\
 3& \text{if $\ell>m$}
 \end{cases}$$
 so for $\ell\le m$, we have $\mu((m),\kn)=\ell$ by Theorem~\ref{thm:mu}.
 Considering conditions \ref{g1}, \ref{g2} and~\ref{g3} in Theorem \ref{thm:general}, we get the necessary and sufficient conditions for $q=2$ as follows:
\ref{g1} covers the case when $m\geq 2$ and $\isa{\n}=0$, and demands $\ell \leq m$,
\ref{g2} applies when $m=1$ and in this case we get $\ell=1$,  finally, \ref{g3} gives three conditions for the case $m \geq 2$ and $\isa{\n}>0$: $1+\ell \leq m$, $2 \leq\ell\leq m$ or $\ell \leq 2 \leq |\n|$, that together reduce to  $ \ell \leq m$.

%
\end{example}

Throughout this section we were noting the effect that the number of connected components and isolated points have on $q$. The example below exposes this behaviour on a specific example.
%
%
%

\begin{example}\label{ex:discrete}
Let $a,b,s\in \bN$. Then:
 \[ q(a K_s \vee b K_1)=
    \begin{cases}
      2& \text{if $s=2 \text{ and } b \in \{a,2a\}, \text{ or } s\neq 2 \text{ and } a \leq b \leq sa$},\\
      3&\text{otherwise}
    \end{cases}\]
and  in the cases that $q(aK_s\vee bK_1)=2$, we have
  $\mu((s^a), (1^b))= b$ by Theorem~\ref{thm:mu}.
  
 To see how the formula for $q(aK_s\vee bK_1)$ follows from conditions \ref{g1}, \ref{g2} and \ref{g3} in Theorem~\ref{thm:general}, we split into various cases. If $s=1$, then condition \ref{g2} gives $q=2$ if and only if $a=b$, as required. If $a\ge b$, then condition \ref{g2} gives $q=2$ if and only if $a=b$, as required.
%
%
In the remaining case, when $b>a$ and $s>1$, the necessary and sufficient conditions for $q=2$ are listed in item (c). In the notation of this example, $q=2$ precisely when at least one of the following conditions  holds:
\begin{align*}
a+b \leq as,  \text{ or } 2 a \leq b \leq as, \text{ or } b=2 a.
\end{align*}
If $s=2$, it is apparent that $b=2a$ is the only solution. For $s\ge3$, we get that $q=2$ if any only if $b \in \{a+1,a+2, \ldots, s a\}$.

\end{example}

\section*{Acknowledgments}
Polona Oblak acknowledges the financial support from the Slovenian Research Agency  (research core funding No. P1-0222).

 \bibliographystyle{amsplain}
\bibliography{qqCref}

\end{document}